\documentclass[11pt]{article}

\usepackage{amssymb,amsmath}
\usepackage{amsthm}
\usepackage{hyperref}
\usepackage{mathrsfs}
\usepackage{textcomp}
\usepackage{verbatim}

\usepackage{sectsty}

\makeatletter

\newdimen\bibspace
\renewenvironment{thebibliography}[1]{%
 \section*{\refname 
       \@mkboth{\MakeUppercase\refname}{\MakeUppercase\refname}}%
     \list{\@biblabel{\@arabic\c@enumiv}}%
          {\settowidth\labelwidth{\@biblabel{#1}}%
           \leftmargin\labelwidth
           \advance\leftmargin\labelsep
           \itemsep\bibspace
           \parsep\z@skip     %
           \@openbib@code
           \usecounter{enumiv}%
           \let\p@enumiv\@empty
           \renewcommand\theenumiv{\@arabic\c@enumiv}}%
     \sloppy\clubpenalty4000\widowpenalty4000%
     \sfcode`\.\@m}
    {\def\@noitemerr
      {\@latex@warning{Empty `thebibliography' environment}}%
     \endlist}

\makeatother

\makeatletter

\newtheorem{thm}{Theorem}[section]
\newtheorem{lem}{Lemma}[section]
\newtheorem{prop}{Proposition}[section]

\newtheorem{cor}{Corollary}[section]
\newtheorem{rem}{Remark}[section]

\theoremstyle{definition}
\newtheorem*{notations*}{Notations}

\def\Xint#1{\mathchoice
  {\XXint\displaystyle\textstyle{#1}}%
  {\XXint\textstyle\scriptstyle{#1}}%
  {\XXint\scriptstyle\scriptscriptstyle{#1}}%
  {\XXint\scriptscriptstyle\scriptscriptstyle{#1}}%
  \!\int}
\def\XXint#1#2#3{{\setbox0=\hbox{$#1{#2#3}{\int}$}
  \vcenter{\hbox{$#2#3$}}\kern-.5\wd0}}

\def\dashint{\Xint-}

\newcommand{\al}{\alpha}                \newcommand{\lda}{\lambda}
\newcommand{\om}{\Omega}                \newcommand{\pa}{\partial}
\newcommand{\va}{\varepsilon}           \newcommand{\ud}{\mathrm{d}}
\newcommand{\be}{\begin{equation}}      \newcommand{\ee}{\end{equation}}
                 
\newcommand{\Lda}{\Lambda}              \newcommand{\B}{\mathcal{B}}
\newcommand{\R}{\mathbb{R}}

\newcommand{\dlim}{\displaystyle\lim}

\title{\textbf{Sharp constants in weighted trace inequalities on Riemannian manifolds}}

\author{\medskip Tianling Jin and Jingang Xiong}

\begin{document}

\maketitle

\begin{abstract}
We establish some sharp weighted trace inequalities $W^{1,2}(\rho^{1-2\sigma}, M)\hookrightarrow L^{\frac{2n}{n-2\sigma}}(\pa M)$
on $n+1$ dimensional compact smooth manifolds with smooth boundaries, where $\rho$ is a defining function of $M$ and $\sigma\in (0,1)$. This is stimulated by some recent work on fractional (conformal) Laplacians and related problems in conformal geometry, and also motivated by a conjecture of Aubin.
\end{abstract}

\section{Introduction}

Let $\om$ be an open set in $\R^n$, $n\geq 1$,  and $\rho(x)=\mathrm{dist}(x,\pa \om)$ for $x\in \om$.   There have been much work devoted to the structures of weighted
Sobolev spaces of the type $W^{k,p}(\rho^{\al}, \om)$ where $\al\in \R$, $k\in \mathbb{N}$ and $1\leq p\leq \infty$, 
as well as to their applications in different areas such as (stochastic) partial differential equations and Riemannian manifolds with fractal boundaries or boundary singularities. We refer to the book \cite{Maz} of Maz'ya and references therein for these topics. 

In this paper, we would like to study sharp constants in weighted trace type inequalities $W^{1,2}(\rho^{1-2\sigma})\hookrightarrow L^{\frac{2n}{n-2\sigma}}(\pa M)$ on Riemannian manifolds $M$ with boundaries $\pa M$. Let us start from Euclidean spaces. Denote $\dot{H}^{\sigma}(\R^n)$ as the $\sigma$-order homogeneous Sobolev space on $\R^n$, $n\geq 2$, which is the closure of $C^\infty_c(\R^n)$ under the norm
\[
\|f\|_{\dot{H}^{\sigma}(\R^n)}=\left(\int_{\R^n}|(-\Delta)^{\sigma/2}f(x)|^2\,\ud x\right)^{1/2}.
\]
The sharp $\sigma$-order Sobolev inequality asserts that 
\[
 \|f\|^2_{L^{\frac{2n}{n-2\sigma}}(\R^n)}\leq c(n,\sigma)\|f\|^2_{\dot{H}^{\sigma}(\R^n)}
\]
for all $f\in  \dot{H}^{\sigma}(\R^n)$, where
\[
 c(n,\sigma)=2^{-2\sigma}\pi^{-\sigma}\left(\frac{\Gamma((n-2\sigma)/2)}{\Gamma((n+2\sigma)/2)}\right)
\left(\frac{\Gamma(n)}{\Gamma(n/2)}\right)^{\frac{2\sigma}{n}},
\]
and the equality holds if and only if $f(x)$ takes the form
\[
 c\left(\frac{\lda}{1+\lda^2|x-x_0|^2}\right)^{\frac{n-2\sigma}{2}}
\]
for some $c\in \R$, $\lda >0$ and $x_0\in \R^n$. These have been proved by Lieb in \cite{Lie83}. Set $x=(x',x_{n+1})\in \R^{n+1}_+:= \R^n\times (0,\infty)$ and
\[
 F(x',x_{n+1})=\int_{\R^n}\mathcal{P}_{\sigma}(x'-\xi, x_{n+1})f(\xi)\,\ud \xi,
\]
where
\be\label{eq:poisson}
 \mathcal{P}_{\sigma}(x', x_{n+1})=\beta(n,\sigma)\frac{x_{n+1}^{2\sigma}}{(|x'|^2+x_{n+1}^2)^{\frac{n+2\sigma}{2}}}
\ee
with the normalization constant $\beta(n,\sigma)>0$ such that $\int_{\R^n} \mathcal{P}_{\sigma}(x', 1)\,\ud x'=1$.
Then one has (see, e.g., \cite{CS})
\[
N_\sigma  \int_{\R^{n+1}_+}x_{n+1}^{1-2\sigma}|\nabla F(x',x_{n+1})|^2\,\ud x=\|f\|^2_{\dot{H}^{\sigma}(\R^n)},
\]
where $N_\sigma=2^{2\sigma-1}\Gamma(\sigma)/\Gamma(1-\sigma)$. Hence, we have
\be\label{sharp trace-1}
 \|f\|^2_{L^{\frac{2n}{n-2\sigma}}(\R^n)}\leq S(n,\sigma)\int_{\R^{n+1}_+}x_{n+1}^{1-2\sigma}|\nabla F(x',x_{n+1})|^2\,\ud x
 \ee
for all $f\in  \dot{H}^{\sigma}(\R^n)$, where $S(n,\sigma)=N_\sigma\cdot c(n,\sigma)$.
Consequently, one can show (see, e.g., Proposition \ref{prop: 1} below together with a density argument) that
\be\label{sharp trace}
 \|U(\cdot, 0)\|^2_{L^{\frac{2n}{n-2\sigma}}(\R^n)}\leq S(n,\sigma)\int_{\R^{n+1}_+}x_{n+1}^{1-2\sigma}|\nabla U(x',x_{n+1})|^2\,\ud x
 \ee
for all $U\in  W^{1,2}(x_{n+1}^{1-2\sigma},\R_+^{n+1})$, which is the closure of $C_c^{\infty}(\overline \R^{n+1}_+)$ under the norm
\[
 \|U\|_{W^{1,2}(x_{n+1}^{1-2\sigma}, \R_+^{n+1})}=\sqrt{\int_{\R^{n+1}_+}x_{n+1}^{1-2\sigma}(|U|^2+|\nabla U|^2)\,\ud x}.
\]

Stimulated by several recent work on fractional (conformal) Laplacians and related problems in conformal geometry (see, e.g., \cite{GZ, CG, GQ, JLX1}) and a conjecture of Aubin \cite{Aubin76}, we study weighted Sobolev trace inequalities of type \eqref{sharp trace} on Riemannian manifolds with boundaries. 
For $n\geq 2$,  let $(M,g)$ be an $n+1$ dimensional, compact, smooth Riemannian manifold with smooth boundary $\pa M$.
We say a function $\rho\in C^\infty(\overline{M})$ is a \emph{defining function} of $M$ if
\[
 \rho>0 \quad \mbox{in }M, \quad \rho=0 \mbox{ and } \nabla_g \rho \neq 0\quad \mbox{on }\pa M.
\]
Since $\rho^{1-2\sigma}$, where $\sigma\in (0,1)$ is a constant, belongs to the
Muckenhoupt $A_2$ class, we define the weighted Sobolev space $H^{1}(\rho^{1-2\sigma}, M)$ as the closure of $C^\infty(\overline{M})$ under the norm
\[
 \|u\|_{H^{1}(\rho^{1-2\sigma}, M)}=\left(\int_{M}\rho^{1-2\sigma}(|u|^2+|\nabla u|^2)\,\ud v_g\right)^{\frac 12},
\]
where $\ud v_g$ denote the volume form of $(M,g)$. $H^{1}(\rho^{1-2\sigma}, M)$ is a Hilbert space and it has a well-defined
\emph{trace operator} $T$ (see, e.g., \cite{Maz} or \cite{N}) which continuously maps  $H^1(\rho^{1-2\sigma}, M)$ to $H^\sigma(\pa M)$, where $H^\sigma(\pa M)$ is the
$\sigma$-order Sobolev space on $\pa M$.

\begin{thm}\label{thm: main thm A}
For $n\geq 2$,  let $(M,g)$ be an $n+1$ dimensional, compact, smooth Riemannian manifold with smooth boundary $\pa M$.
Let $\sigma\in (0,\frac{1}{2}]$, and $\rho$ be a defining function of
$M$ satisfying $|\nabla_g \rho|=1$ on $\pa M$. Then there exists a positive constant $A=A(M,g,n,\rho,\sigma)$ such that
\be\label{main ineq}
\left(\int_{\pa M}|u|^{\frac{2n}{n-2\sigma}}\,\ud s_g\right)^\frac{n-2\sigma}{n}\leq S(n,\sigma)\int_{M}
\rho^{1-2\sigma}|\nabla_g u|^2\,\ud v_g+A \int_{\pa M}u^2\,\ud s_g,
\ee
for all $u\in H^1(\rho^{1-2\sigma}, M)$, where $\ud s_g$ denotes the induced volume form on $\pa M$.
\end{thm}

 For $\sigma\in (\frac{1}{2},1)$, we have

\begin{thm}\label{thm: main thm B} Let $\sigma\in (\frac{1}{2},1)$, $n\geq 4$ and $(M,g)$ be an $n+1$ 
dimensional, compact, smooth Riemannian manifold with smooth boundary $\pa M$. Suppose in addition that $\pa M$ is totally geodesic. 
Let $\rho$ be a defining function of
$M$ satisfying $\rho(x)=d(x)+O(d(x)^3)$ as $d(x)\to 0$, where $d(x)$ denotes the distance between $x$ and $\pa M$ with respect to the metric $g$.
Then there exists a positive constant $A=A(M,g,n,\rho,\sigma)$
such that \eqref{main ineq} holds for all $u\in H^1(\rho^{1-2\sigma}, M)$.
\end{thm}

\begin{rem}
 The constant $S(n,\sigma)$ in \eqref{main ineq} is optimal for all $\sigma\in (0,1)$, see Proposition \ref{prop: mini const}.
\end{rem}

\begin{rem}
Theorem \ref{thm: main thm B} may fail without any geometric assumption on $\pa M$. For example, it is the case when the mean curvature of $\pa M$ is positive somewhere. In particular, \eqref{main ineq} is false on any bounded smooth domain in $\R^{n+1}$ when $\sigma\in (1/2,1)$. However, Theorem \ref{thm: main thm A} holds for all $\sigma\in (0,1)$ if $S(n,\sigma)$ is replaced by any $S>S(n,\sigma)$, see Proposition \ref{prop: weak version ineq}.
\end{rem}

\begin{rem}
It is clear that we only need to consider the case when $M$ is connected.
Throughout the paper, we assume this.
\end{rem}

When $\sigma=\frac{1}{2}$, \eqref{main ineq} is a standard Sobolev trace inequality
which has been extensively studied, see, e.g.,  Lions \cite{Lions85}, Escobar \cite{Escobar88},
Beckner \cite{Be}, Adimurthi-Yadava \cite{AY}, Li-Zhu \cite{LZ2, LZ1} and many others.
In particular, Li-Zhu \cite{LZ2} established Theorem \ref{thm: main thm A} for $\sigma =\frac{1}{2}$.
The sharp inequality \eqref{main ineq} is in the same spirit of a conjecture posed by Aubin \cite{Aubin76}
which concerns the best constants in Sobolev embedding theorems on Riemannian manifolds. Aubin's
conjecture had been confirmed through the work of  Hebey-Vaugon \cite{HV}, Aubin-Li \cite{AL} and Druet \cite{Druet99, Druet02}.
Besides, various refinements of Aubin's conjecture were obtained in Druet-Hebey \cite{DH}, Li-Ricciardi \cite{LR} and etc. These sharp Sobolev type inequalities play important roles in the study of nonlinear partial differential equations,
see Aubin \cite{Aubin98}, Hebey \cite{H}, Schoen-Yau \cite{SY} and references therein.

For the defining function in the above theorems, $(M,g/\rho^2)$ is \emph{asymptotically hyperbolic} in the sense that
$(M,g/\rho^2)$ is a complete manifold and along any
smooth curve in $M\setminus \pa M$ tending to a point $\xi \in \pa M$
all sectional curvatures of $g/\rho^2$ approach to $-1$ (see Mazzeo \cite{M88} or Mazzeo-Melrose \cite{MM87}).
On the conformal infinity $(\pa M, [g|_{\pa M}])$ of $(M,g/\rho^2)$, one can define fractional order conformally invariant
operators $P_{\sigma}^g$ for $\sigma\in (0,\frac{n}{2})$ except at most finite values, via normalized scattering operators (see  Graham-Zworski \cite{GZ} and Chang-Gonz\'alez \cite{CG}), which leads to $\sigma$-scalar curvature $R_\sigma^g:=P_{\sigma}^g(1)$ on $\pa M$. A fractional Yamabe problem, which is to find a metric in $[g|_{\pa M}]$ of constant $\sigma$-curvature and related ones, have been studied by Qing-Raske \cite{QR}, Gonz\'alez-Mazzeo-Sire \cite{GMS} and Gonz\'alez-Qing \cite{GQ}. When $\sigma\in (0,1)$, it can be formulated (see \cite{GQ}) as seeking minimizers of the energy functional
\be \label{vp}
I^\sigma[u]=\frac{N_\sigma \int_{M}\rho^{1-2\sigma}|\nabla u|^2\,\ud v_g+\int_{\pa M} R_\sigma^gu^2\,\ud s_g}{\big(\int_{\pa M} |u|^\frac{2n}{n-2\sigma}\,\ud s_g\big)^{\frac{n-2\sigma}{n}}}, \quad u\in H^1(\rho^{1-2\sigma}, M)\ ,\ u\not\equiv 0\ \mbox{on}\ \pa M,
\ee
for some proper $\rho$. For $\sigma=1/2$, it is the energy functional of a Yamabe problem with boundary initially studied by Escobar \cite{E2}.
A fractional Nirenberg problem about prescribing $\sigma$-scalar curvature on $\mathbb{S}^n$  has been studied by  Jin-Li-Xiong \cite{JLX1,JLX2} and a fractional Yamabe flow has been studied by Jin-Xiong \cite{JX}. Variational problems related to energy functional \eqref{vp} on bounded domains in Euclidean spaces have been studied by Gonz\'alez \cite{G}, Palatucci-Sire \cite{PS}.

Finally, we provide a brief sketch of the proofs of the two main theorems. Since the right hand side of \eqref{main ineq} does not contain terms like
$\int_{M}\rho^{1-2\sigma} u^2\,\ud v_g$, we adapt a global argument from Li-Zhu \cite{LZ2,LZ1}.
By contradiction, we assume that for any $\al >0$,
\[
 I_\al :=\frac{\int_{M}\rho^{1-2\sigma}|\nabla_g u|^2\,\ud v_g+\al\int_{\pa M}|u|^2\,\ud s_g}{\big(\int_{\pa M} |u|^\frac{2n}{n-2\sigma}\,\ud s_g\big)^{\frac{n-2\sigma}{n}}}<\frac{1}{S(n,\sigma)},
\]
for some $u\in H^1(\rho^{1-2\sigma},  M)$ with that $u\not\equiv 0$ on $\pa M$. It follows that there exists a minimizer $u_\al$ of $I_\al$, and
$u_\al$ blows up at exactly one point as $\al \to \infty$. One key step is the asymptotical analysis of $u_\al$ near its blow up point.
Here we have to overcome difficulties from the degeneracy and the lack of conformal invariance of the Euler-Lagrange equation of $I_\al$ satisfied by $u_\al$.
Another difference from \cite{LZ2} (the case $\sigma=1/2$) is that some Sobolev embedding theorems for $H^1(\rho^{1-2\sigma}, M)$, which play important roles in establishing the blow-up profile of $u_\al$ in the interior of $M$ in \cite{LZ2} in the case $\sigma=\frac 12$, fail when $\sigma>\frac 12$ (see, e.g., Theorem 1 in page 135 or Corollary 2 in page 193 of \cite{Maz}) .  However, we succeeded in establishing the optimal asymptotical behavior of $u_\al$ on the boundary $\pa M$ (Proposition \ref{prop: up bound on the boundary}).
In this step, a Liouville type theorem in Jin-Li-Xiong \cite{JLX1} and  \emph{Neumann functions} for degenerate equations in Theorem \ref{thm:neumann} are used.
The last step is to derive a contradiction by checking balance via a Pohozaev type inequality in some proper region, where a Harnack inequality established by Cabre-Sire \cite{CaS} or Tan-Xiong \cite{TX} is used
to obtain the asymptotical behavior of $u_\al$ near it blowup point in $M$ from that on $\pa M$. Some extra arguments on $\pa M$ are needed for $\sigma >\frac{1}{2}$. 

\begin{thm}\label{thm:neumann} Let $f\in L^1(\pa M)$ with mean value zero, i.e., $\int_{\pa M} f=0$. Then there exists a weak solution $u\in W^{1,1+\va_0}(\rho^{1-2\sigma},M)$ of \eqref{nb1} where $\va_0>0$ depending only on $n$ and $\sigma$. Consequently, if $f=\delta_{x_0}-\frac{1}{|\pa M|_g}$ for some $x_0\in \pa M$, where $\delta_{x_0}$ is the delta function at $x_0$ and $|\pa M|_g$ is the area of $\pa M$ with respect to the induced metric $g$, then there exists a weak solution $u\in W^{1,1+\va_0}(\rho^{1-2\sigma},M)\cap H^1_{loc}(\rho^{1-2\sigma},\overline M\setminus\{x_0\})$ of \eqref{nb1} with mean value zero. Moreover, for all $x\in \overline{M}\backslash\{x_0\}$,
\[
 A_1 \mathrm{dist}_g(x, x_0)^{2\sigma-n}-A_0\leq u(x)\leq A_2 \mathrm{dist}_g(x, x_0)^{2\sigma-n}
 \]
where $A_0, A_1, A_2$ are positive constants depending only on $M, g, n, \sigma, \rho$.
\end{thm}
The proof of Theorem \ref{thm:neumann} follows from Lemma \ref{lem:nb3}, Theorem \ref{thm:nb2} and some approximation arguments. When $\sigma=1/2$, Theorem \ref{thm:neumann} follows directly from Brezis-Strauss \cite{BS} and Kenig-Pipher \cite{KP}.

\begin{notations*} 
We collect below a list of the main notations used throughout the paper.
\begin{itemize}
\item We always assume that $n\ge 2, \sigma\in (0,1)$, and $\rho$ is a smooth defining function as in Theorem \ref{thm: main thm A} without otherwise stated. Denote $q=\frac{2n}{n-2\sigma}$.

\item For a domain $D\subset \R^{n+1}$ with boundary $\pa D$, we denote $\pa' D$ as the interior of $\overline D\cap \pa \R^{n+1}_+$ in $\R^n=\partial\R^{n+1}_+$ and $\pa''D=\pa D \setminus \pa' D$.

\item For $\bar x\in \R^{n+1}$, $\B_{r}(\bar x):=\{x\in \R^{n+1}: |x-\bar x|=\sqrt{(x_1-\bar x_1)^2+\cdots+(x_{n+1}-\bar x_{n+1})^2}<r\}$, $\B^+_{r}(\bar x):=\B_{r}(\bar x)\cap \R^{n+1}_+$. If $\bar x\in \pa\R^{n+1}_+$,  $B_{r}(\bar x):=\{x=(x', 0): |x'-\bar x'|<r\}$. Hence $\pa' \B^+_{r}(\bar x)=B_{r}(\bar x)$ if $\bar x\in \pa\R^{n+1}_+$. We will not keep writing the center $\bar x$ if $\bar x=0$. 

\end{itemize}

\end{notations*}
\bigskip

\noindent\textbf{Acknowledgements:} Both authors thank Prof. Y.Y. Li for encouragements and useful discussions. Tianling Jin was partially supported by a University and Louis Bevier Dissertation Fellowship at Rutgers University and Rutgers University School of Art and Science Excellence Fellowship. Jingang Xiong was partially supported by CSC project for visiting Rutgers University and NSFC No. 11071020. He is very grateful to the Department of Mathematics at Rutgers University for the kind hospitality.

\section{Preliminaries}
\label{sec: preliminary}

\begin{prop}\label{prop: 1} For any $u\in C^\infty_c(\overline{\R}^{n+1}_+)$, we have
\[
 \left(\int_{\R^n}|u(x',0)|^q\,\ud x' \right)^\frac{2}{q}\leq S(n,\sigma) \int_{\R^{n+1}_+}x_{n+1}^{1-2\sigma}|\nabla u(x)|^2\,\ud x.
\]
Moreover, the above inequality fails if $S(n,\sigma)$ is replaced by any smaller constant.
\end{prop}
\begin{proof} It follows from \eqref{sharp trace} and Lemma A.3 of \cite{JLX1}.  
See also Corollary 5.3 of \cite{GQ}.
\end{proof}

\begin{prop} \label{prop: mini const}
Let $M$ be as in Theorem \ref{thm: main thm A}. Let $\sigma\in (0,1)$, and $\rho$ be a defining function of
$\pa M$ with $|\nabla_g \rho|=1$ on $\pa M$. Suppose there exist some positive constants $\tilde{S}$ and $\tilde{A}$
such that, for all $u\in H^1(\rho^{1-2\sigma}, M)$,
\[
\left(\int_{\pa M}|u|^q\,\ud s_g\right)^\frac{2}{q}\leq \tilde{S}\int_{M}\rho^{1-2\sigma}|\nabla_g u|^2\,\ud v_g+
\tilde{A} \int_{\pa M}|u|^2\,\ud s_g.
\]
Then $\tilde{S}\geq S(n,\sigma)$.
\end{prop}

\begin{proof}
Given Proposition \ref{prop: 1}, the proof is standard (see, e.g.,  Proposition 4.2of \cite{H}). We include it here for completeness and to illustrate the role of $|\nabla\rho|=1$.
 We argue by contradiction. Suppose that there exists a Riemannian manifold $(M,g)$, a defining function $\rho$ of $\pa M$ with $|\nabla_g \rho|=1$ on $\pa M$, $\sigma \in (0,1)$, $\tilde S<S(n,\sigma)$ and $\tilde A>0$ such that for all $u\in H^1(\rho^{1-2\sigma}, M)$,
 \be\label{eq:smaller}
\left(\int_{\pa M}|u|^q\,\ud s_g\right)^\frac{2}{q}\leq \tilde{S}\int_{M}\rho^{1-2\sigma}|\nabla_g u|^2\,\ud v_g+
\tilde{A} \int_{\pa M}|u|^2\,\ud s_g.
\ee
Let $x\in \pa M$. For any $\va>0$, which will be chosen sufficiently small, there exists a chart $(\Omega, \varphi)$ of $M$ at $x$ and $\delta>0$ such that $\varphi(\Omega)=\mathcal{B}^+_{\delta}(0)$ the upper half Euclidean ball of center $0$ and radius $\delta$ in $\R^{n+1}_+$, and
\be\label{metric small}
(1-\va)\delta_{ij}\le g_{ij}\le (1+\va)\delta_{ij}.
\ee 
 By assumption, \eqref{eq:smaller} holds for any $u\in C_c^{\infty}(\Omega\cup(\pa\Omega\cap\pa M))$, i.e.,
 \[
 \begin{split}
 \left(\int_{B_{\delta}(0)}|u|^q \sqrt{\det(g_{ij})}\,\ud x'\right)^\frac{2}{q}&\leq \tilde{S}\int_{\B_{\delta}^+(0)}\rho^{1-2\sigma}g^{ij}u_i u_j \sqrt{\det(g_{ij})}\,\ud x\\
&\quad +\tilde{A} \int_{B_{\delta}(0)}|u|^2\sqrt{\det(g_{ij})}\,\ud x'.
\end{split}
 \]
It follows from \eqref{metric small}, $|\nabla_g \rho|=1$ and $\rho=0$ on $\pa M$ that there exists $\delta_0>0, \tilde S'<S(n,\sigma), \tilde A'>0$ such that for all $\delta\in (0,\delta_0)$ and $u\in C_c^{\infty}(\B_{\delta}(0)\cup B_{\delta}(0))$, i.e.,
 \[
 \left(\int_{B_{\delta}(0)}|u|^q \,\ud x'\right)^\frac{2}{q}\leq \tilde{S}'\int_{\B_{\delta}^+(0)}x_{n+1}^{1-2\sigma}|\nabla u|^2\,\ud x+\tilde{A}' \int_{B_{\delta}(0)}|u|^2\,\ud x'.
 \]
By H\"older's inequality, $\int_{B_{\delta}(x)}|u|^2\,\ud x'\leq |B_{\delta}(0)|^{\frac{q-2}{q}}\left(\int_{B_{\delta}(0)}|u|^q \,\ud x'\right)^\frac{2}{q}$. By choosing $\delta$ sufficiently small, we have that there exists $\tilde S''<S(n,\sigma)$ such that for all $u\in C_c^{\infty}(\B_{\delta}(0)\cup B_{\delta}(0))$
\[
 \left(\int_{B_{\delta}(0)}|u|^q \,\ud x'\right)^\frac{2}{q}\leq \tilde{S}''\int_{\B_{\delta}^+(0)}x_{n+1}^{1-2\sigma}|\nabla u|^2\,\ud x.
 \]
 Consequently, by a scaling argument, we have
 \[
 \left(\int_{\R^n}|u(x',0)|^q\,\ud x' \right)^\frac{2}{q}\leq \tilde{S}'' \int_{\R^{n+1}_+}x_{n+1}^{1-2\sigma}|\nabla u(x)|^2\,\ud x.
\]
 for any $u\in C^\infty_c(\overline{\R}^{n+1}_+)$, which contradicts Proposition \ref{prop: 1}.
\end{proof}

\begin{prop}\label{prop: weak ineq}  Assume the assumptions in  Proposition \ref{prop: mini const}. Then
for any $\va>0$ there exists a positive constant $B_\va$ such that
\[
\left(\int_{\pa M}|u|^q\,\ud s_g\right)^\frac{2}{q}\leq (S(n,\sigma)+\va)\int_{M}\rho^{1-2\sigma}|\nabla_g u|^2\,\ud v_g
+B_\va \int_{ M}\rho^{1-2\sigma}|u|^2\,\ud v_g.
\]
\end{prop}

\begin{proof}
 It also follows from Proposition \ref{prop: 1}  and a standard partition of unity argument, see, e.g., Theorem 4.5 of \cite{H} on page 95.
\end{proof}

For every $\al >0$, consider the functional
\[
 I_\al[u]=\frac{\int_{M}\rho^{1-2\sigma}|\nabla_g u|^2\,\ud v_g+\al\int_{\pa M}|u|^2\,\ud s_g}{\left(\int_{\pa M}|u|^q\,\ud s_g\right)^{2/q}},
\quad u\in H^1(\rho^{1-2\sigma}, M),\ \ u\not\equiv 0\ \mbox{on}\ \pa M.
\]
\begin{prop}\label{prop: existence of mini}
Suppose that for some $\al >0$,
\be\label{contradict assup}
\xi_\al:=\inf_{u\in H^1(\rho^{1-2\sigma}, M), \ u|_{\pa M}\not\equiv 0}I_\al[u]<\frac{1}{S(n,\sigma)},
\ee then $\xi_\al$ is achieved by a nonnegative function $u_\al \in H^1(\rho^{1-2\sigma}, M)$ with
\be\label{unit integral}
\int_{\pa M}u_\al^q\,\ud s_g=1.
\ee
\end{prop}

\begin{proof}
Given Proposition \ref{prop: weak ineq}, the Proposition follows from standard calculus of variations, see page 452 of \cite{LZ2}.
\end{proof}

\begin{prop}\label{prop: weak version ineq}  Assume the assumptions in  Proposition \ref{prop: mini const}.
For any $\va>0$, there exists a positive constant $A_\va$ such that
\[
 \left(\int_{\pa M}|u|^q\,\ud s_g\right)^\frac{2}{q}\leq (S(n,\sigma)+\va)\int_{M}\rho^{1-2\sigma}|\nabla_g u|^2\,\ud v_g+A_\va \int_{\pa M}|u|^2\,\ud s_g.
\]
\end{prop}
\begin{proof}
Given Propositions \ref{prop: weak ineq} and \ref{prop: existence of mini}, and Corollary \ref{SSE on manifold}, the proof of Proposition \ref{prop: weak version ineq} is similar to Proposition 1.2 of \cite{LZ2} and we omit it here.
\end{proof}

\section{Asymptotic analysis}
\label{sec: asymp anal}

For brevity, from now on we write $S$ instead of $S(n,\sigma)$. We prove Theorem \ref{thm: main thm A} by contradiction.
 Namely, assume that for any $\al \geq 1$,
\be\label{3-1}
 \xi_\al<\frac{1}{S},
\ee
where $\xi_\al$ is defined as in Proposition \ref{prop: existence of mini}.
Let $u_\al$ be some nonnegative minimizer of $I_\al$ obtained in
Proposition \ref{prop: existence of mini} which satisfies
\be\label{eq: minimizer integral norm}
\xi_{\al}=\int_{M}\rho^{1-2\sigma}|\nabla_g u_{\al}|^2\,\ud v_g+\al\int_{\pa M} u_{\al}^2\,\ud s_g,\quad \int_{\pa M} u_{\al}^q\,\ud s_g=1,
\ee
and for any $\varphi\in H^1(\rho^{1-2\sigma}, M)$,
\be\label{eq:weak solution}
\int_{M}\rho^{1-2\sigma}\langle\nabla_g u_{\al}, \nabla_g\varphi\rangle_g\,\ud v_g+\al\int_{\pa M} u_{\al}\varphi\,\ud s_g=\xi_\al\int_{\pa M} u_{\al}^{q-1}\varphi\,\ud s_g.
\ee

The geodesic distance function $d(x):=\mathrm{dist}(x,\pa M)$ determines for some $\va_0>0$ an identification of $\pa M\times [0,\va_0)$ with a neighborhood of $\pa M$ in $M$: $(x', d)\in \pa M\times [0,\va_0)$ corresponds to the point obtained by following the integral curve of $\nabla _g d$ emanating from $x'$ for $d$ units of time. 
Furthermore, $\nabla_g d$ is orthogonal to the slices $\pa M\times \{d\}$. Define $\nu:= -\nabla _g d$ for $d<\va_0$. It follows from Theorem \ref{thm: L-infty}, Theorem \ref{thm: schauder} and Proposition \ref{thm: schauder-1} that $u_\al \in C^\gamma(\overline{M})\cap C^{\infty}(M)\cap C^\infty(\pa M)$  for some $\gamma\in (0,1)$ and
$\rho^{1-2\sigma}\frac{\pa_g u_\al}{\pa \nu}\in C(\pa M\times[0, \va_0/2])$. Hence, $u_\al$ satisfies the Euler-Lagrange equation 
\be\label{E-L eqn}
\begin{cases}
 \operatorname{div}_{g}\Big(\rho^{1-2\sigma}\nabla_g u_\al\Big)=0,&\quad \mbox{in }M,\\
\dlim_{d\to 0}\rho^{1-2\sigma}(x', d)\frac{\pa_g u_{\al}}{\pa \nu}(x', \rho)=\xi_\al u_\al^{q-1}(x')-\al u_\al(x'),& \quad \mbox{on }\pa M.
\end{cases}
\ee
in the pointwise sense. 

 It follows from the maximum  principle that
$\max_{\overline M}u_\al=\max_{\pa M} u_\al$. Let
$u_\al(x_\al)=\max_{\overline M}u_\al$, where $x_\al\in \pa M$, and $\mu_\al=u_\al(x_\al)^{-\frac{2}{n-2\sigma}}$.
By a Hopf Lemma (see, e.g., Proposition 4.11 in \cite{CaS}),
we have $\xi_\al u_\al(x_\al)^{q-1}-\al u_\al(x_\al)>0$, that is
\be\label{almu}
\al \mu_\al^{2\sigma} < \xi_\al.
\ee
Hence, $\lim_{\al \to \infty}\mu_\al^{2\sigma} = 0$.

\begin{lem}
As $\al \to \infty$, we have
\begin{subequations}
\begin{align}
 \xi_\al &\to \frac{1}{S}, \label{sub a}\\
\al \|u_\al\|^2_{L^2(\pa M)}&\to 0.\label{sub b}
\end{align}
\end{subequations}
\end{lem}
\begin{proof}
For all small $\va>0$, it follows from Proposition \ref{prop: weak version ineq} that
\[
\begin{split}
1&\leq (S+\va)\int_{M}\rho^{1-2\sigma}|\nabla_g u_\al|^2\,\ud v_g+A_\va \int_{\pa M} u_\al^2\,\ud s_g\\&
=(S+\va)\xi_\al +(A_\va-(S+\va)\al)\int_{\pa M} u_\al^2\,\ud s_g.
\end{split}
\]
Hence, for every $\al\geq \frac{2A_\va}{S+\va}$ we have
\[
\frac{1}{S+\va}\leq \xi_\al<\frac{1}{S},\quad \frac{S }{2}\al\int_{\pa M} u_\al^2\,\ud s_g< \frac{\va}{S}.
\]
\eqref{sub a} and \eqref{sub b} follow immediately.
\end{proof}

Let $x=(x_1,\cdots, x_n, x_{n+1})=(x',x_{n+1})$ be \emph{Fermi coordinates} (see, e.g., \cite{E2}) at $x_\al$, where $(x_1, \cdots, x_n)$ are normal coordinates on $\pa M$ at $x_\al$ and $\gamma(x_{n+1})$ is the geodesic leaving from $(x_1,\cdot, x_n)$ in the orthogonal direction to $\pa M$ and parametrized by arc length.
In this coordinate system, 
\[
\sum_{1\leq i,j\leq n+1}g_{ij}(x)\ud x_i\ud x_j =\ud x_{n+1}^2+\sum_{1\leq i,j\leq n}g_{ij}(x)\ud x_i\ud x_j.
\]
Moreover, $g^{ij}$ has the following Taylor expansion near $\pa M$:
\begin{lem}[Lemma 3.2 in \cite{E2}]\label{taylor expansion of fermi}
For $\{x_k\}_{k=1,\cdots, n+1}$ are small,
\be\label{eq:taylor expansion of fermi}
g^{ij}(x)=\delta^{ij}+2h^{ij}(x',0)x_{n+1}+O(|x|^2),
\ee
where $i,j=1,\cdots,n$ and $h_{ij}$ is the second fundamental form of $\pa M$. 
\end{lem}
For suitably small $\delta_0>0$ (independent of $\al$), we define $v_\al$ in a neighborhood of
$x_\al=0$ by
\[
 v_\al(x)=\mu_\al^{(n-2\sigma)/2}u_\al(\mu_\al x),\quad x\in \mathcal{B}_{\delta_0/\mu_\al}^+.
\]
It follows that
\be\label{eq 0}
\begin{cases}
 \operatorname{div}_{g_\al}\Big(\rho_\al^{1-2\sigma}\nabla_{g_\al} v_\al\Big)=0,&\quad \mbox{in }\mathcal{B}_{\delta_0/\mu_\al}^+\\
\lim_{x_{n+1}\to 0^+}\rho_\al^{1-2\sigma}\frac{\pa_{g_\al} v_{\al}}{\pa \nu}=\xi_\al v_\al^{q-1}-\al \mu_\al^{2\sigma} v_\al,& \quad \mbox{on }\pa' \mathcal{B}_{\delta_0/\mu_\al}^+=B_{\delta_0/\mu_\al}\\
v_\al(0)=1,\quad  0\leq v_\al \leq 1,
\end{cases}
\ee
where $g_\al(x)=g_{ij}(\mu_\al x)\ud x_i\ud x_j$, $\rho_\al(x)=\rho(\mu_\al x)/\mu_\al$.
It follows from \eqref{almu} and Theorem \ref{thm: L-infty} in the Appendix that
for all $R>1$,
\be\label{holder estimate 1}
\|v_\al\|_{C^\gamma(\mathcal{B}_R^+)}+\|v_\al\|_{H^1(\rho^{1-2\sigma}_\al,\mathcal{B}_R^+)}\leq C(R),\quad \mbox{for all sufficiently large }\al,
\ee
where $\gamma\in (0,1)$ is independent of $R$ and $\al$. It follows that there exists $v\in C_{loc}^\gamma(\overline{\mathbb{R}}^{n+1}_+)
\cap H_{loc}^1(x_{n+1}^{1-2\sigma},\overline{\mathbb{R}}^{n+1}_+)$ such that along some subsequence,
\be\label{holder estimate 3}
\begin{cases}
v_\al&\to v \mbox{ in } C^{\gamma/2}(\mathcal{B}_R^+),\\
v_\al&\rightharpoonup v \mbox{ weakly in } H^1(x_{n+1}^{1-2\sigma},\mathcal{B}_R^+)
\end{cases}
\ee
for any $R>0$ as $\al\to\infty$. Since $v_\al(0)=1$, we have
\be\label{holder estimate 2}
\begin{split}
 &\int_{B_1}v_\al^q\,\ud s_{g_\al}\geq 1/C>0,\\&
\int_{B_1}v_\al^2\,\ud s_{g_\al}\geq 1/C>0.
\end{split}
\ee
On the other hand,
\[
 \al \|u_\al\|^2_{L^2(\pa M)}\geq \al\int_{B_{\mu_\al}(x_\al)}u_\al^2= \al \mu_\al^{2\sigma}\int_{B_1}v_\al^2,
\]
where we abused notation by denoting $B_{r}(x_\al)$ as the geodesic ball on $\pa M$ centered at $x_\al$ with radius $r$.
It follows from \eqref{sub b} and \eqref{holder estimate 2} that
\be\label{almu1}
\lim_{\al \to \infty}\al \mu_\al^{2\sigma}=0.
\ee
From \eqref{eq 0}, \eqref{almu1} and \eqref{sub a}, we conclude that $v$ is a weak solution (see Section \ref{sec of weak solution} for the definition of weak solutions) of
\be\label{limt equ}
\begin{cases}
\operatorname{div}(x_{n+1}^{1-2\sigma}\nabla v)=0,\quad &\mbox{in }\mathbb{R}^{n+1}_+,\\
-\dlim_{x_{n+1}\to 0^+}x_{n+1}^{1-2\sigma}\pa_{x_{n+1}}v=\frac{1}{S}v^{q-1},\quad & \mbox{on }\pa\mathbb{R}^{n+1}_+,\\
v(0)=1,\quad 0\leq v\leq 1.
\end{cases}
\ee
By a Liouville type theorem, Theorem 1.5 in \cite{JLX1},
\[
v(x',0)=\left(\frac{1}{1+\tilde c(n,\sigma)|x'|^2}\right)^{\frac{n-2\sigma}{2}}, v(x', x_{n+1})=\int_{\R^n}\mathcal P_{\sigma}(x'-y', x_{n+1})v(y',0)\ud y',
\]
where $\tilde c(n,\sigma)$ is a positive constant such that $\int_{\R^n}v^q(z)\,\ud z=1$, and $\mathcal P_{\sigma}(x)$ is given in \eqref{eq:poisson}. Due to the uniqueness of the limit function $v$, we know that \eqref{holder estimate 3} holds for all $\al\to\infty$.

\begin{prop}
 \label{prop: integral converge}
For $\delta_0=\delta_0(M,g)>0$ small enough,
\[
 \lim_{\al\to \infty}\int_{B_{\delta_0/\mu_\al}}|v_\al-v|^q=0.
\]
\end{prop}
\begin{proof}
Note that $v_\al \geq 0$ and
\be\label{1 flw integral converge}
\int_{B_{\delta_0/\mu_\al}}v_\al^q\leq \int_{\pa M}u_\al^q=1.
\ee
For any $\va>0$, choose $R>0$ such that $\int_{\R^n\setminus B_R} v^q(x',0)\,\ud x'\le \va$. It follows from \eqref{holder estimate 3} that $\int_{B_R}|v_\al-v|^q\leq \va$ and $1-\int_{B_R}v_{\al}^q<2\va$ for all $\al$ sufficiently large.
Then
\[
\begin{split}
 &\int_{B_{\delta_0/\mu_\al}}|v_\al-v|^q\\&=\int_{B_{\delta_0/\mu_\al}\cap B_R}|v_\al-v|^q+\int_{B_{\delta_0/\mu_\al}\cap B^c_R}|v_\al-v|^q\\&
\leq \int_{B_{\delta_0/\mu_\al}\cap B_R}|v_\al-v|^q+
2^q\int_{B_{\delta_0/\mu_\al}\cap B^c_R}v_\al^q+2^q\int_{B_{\delta_0/\mu_\al}\cap B^c_R}v^q\\&
\leq\va +2^q(1-\int_{B_R}v_{\al}^q)+ 2^q (1-\int_{B_R}v^q)\leq \va(1+3\cdot 2^{q}),
\end{split}
\]
which finishes the proof.
\end{proof}
\begin{cor}
 \label{cor: concentrate} For all $\delta_1>0$ we have
\[
\lim_{\al \to \infty}\int_{B_{\delta_1}(x_\al)\cap \pa M} u_\al^q=1.
\]
\end{cor}
\begin{proof}
It follows immediately from Proposition \ref{prop: integral converge}.
\end{proof}

Let $\tilde G_{\al}$ be the weak solution of
\[
\begin{cases}
-\operatorname{div}_g(\rho^{1-2\sigma}\nabla_g\tilde G_\al)=0, &\quad \mbox{in }M,\\
\dlim_{y\to x\in \pa M} \rho^{1-2\sigma}(y)\frac{\pa }{\pa \nu}\tilde G_\al(y)=\delta_{x_\al}-\frac{1}{|\pa M|_g},&\quad \mbox{on }\pa M,
\end{cases}
\]
constructed in Theorem \ref{thm:nb2}.
We can find a positive constant $C>0$ sufficiently large depending only on $M, g, n, \sigma, \rho$ such that $G_\al:=\tilde G_\al+C\geq 1$ on $\overline M$.

\begin{prop}\label{prop: eqn reduction}
Let $\varphi_\al(x)=\mu_\al^{\frac{n-2\sigma}{2}}G_\al(x)$,
$\tilde{g}_{ij}=\varphi_\al^{\frac{4}{n-2\sigma}}g_{ij}$ and $a=2-\frac{2(n-1)}{n-2\sigma}$. Then $w_\al:=\frac{u_\al}{\varphi_\al}$ satisfies
\be\label{eqn reduction}
\begin{cases}
\operatorname{div}_{\tilde{g}}\Big(\varphi_\al^a\rho^{1-2\sigma}\nabla_{\tilde{g}} w_\al\Big)=0,&\quad \mbox{in } M,\\
\dlim_{y\to \bar x\in \pa M }\varphi_\al^a\rho^{1-2\sigma} \frac{\pa_{\tilde{g}}w_\al(y)}{\pa \tilde{\nu}}\leq \xi_\al w_\al^{q-1}(\bar x),& \quad 
\bar x\in \pa M \setminus \{x_\al \},  
\end{cases}
\ee
for $\al\geq \frac{1}{|\pa M|_g}$.
\end{prop}
\begin{proof} The proof follows from some direct computations. For brevity, we drop the subscript $\al$ of $\varphi_\al$ and $u_\al$.
First of all,
\[
\begin{split}
&\operatorname{div}_{\tilde{g}}\Big(\varphi^a\rho^{1-2\sigma}\nabla_{\tilde{g}} \frac{u}{\varphi}\Big)\\&
=\varphi^{a-1-\frac{4}{n-2\sigma}}\operatorname{div}_{g}\Big(\rho^{1-2\sigma}\nabla_{g} u\Big)-
u\varphi^{a-2-\frac{4}{n-2\sigma}}\operatorname{div}_{g}\Big(\rho^{1-2\sigma}\nabla_{g} \varphi\Big)\\&
\quad +\left(a-2+\frac{2(n-1)}{n-2\sigma}\right)\rho^{1-2\sigma}\varphi^{a-2-\frac{4}{n-2\sigma}}
\left(\langle\nabla_g u, \nabla_g\varphi\rangle_g
-u\varphi |\nabla_g \varphi|^2_{g}\right)\\&
=0.
\end{split}
\]
On the other hand, in Fermi coordinate system centered at $\bar x$, 
\[
\begin{split}
&\lim_{x_{n+1}\to 0} \varphi^a \rho^{1-2\sigma}\frac{\pa_{\tilde{g}}}{\pa \tilde{\nu}}(\frac{u}{\varphi})\\&=
\lim_{x_{n+1}\to 0} \varphi^a \rho^{1-2\sigma}\left(\frac{1}{\varphi}\frac{\pa u}{\pa x_{n+1}}-\frac{u}{\varphi^2}\frac{\pa \varphi}{\pa x_{n+1}}\right)
\tilde{g}^{n+1,n+1}\langle \frac{\pa }{\pa x_{n+1}},\tilde{\nu}\rangle_{\tilde g}\\&
=\varphi^{a-1-\frac{2}{n-2\sigma}} (\xi_\al u^{\frac{n+2\sigma}{n-2\sigma}}-\al u)+\varphi^{a-2-\frac{2}{n-2\sigma}}u\mu_\al^{\frac{n-2\sigma}{2}}\frac{1}{|\pa M|}\\&
\leq \xi_\al\left(\frac{u}{\varphi}\right)^{\frac{n+2\sigma}{n-2\sigma}}+\varphi^{a-2-\frac{2}{n-2\sigma}}
u\mu_\al^{\frac{n-2\sigma}{2}}(\frac{1}{|\pa M|_g}-\al)\\&
\leq \xi_\al\left(\frac{u}{\varphi}\right)^{\frac{n+2\sigma}{n-2\sigma}},
\end{split}
\]
provided  $\al\geq \frac{1}{|\pa M|_g}$.
\end{proof}

\begin{prop}\label{prop: up bound on the boundary}
Suppose the assumptions in Proposition \ref{prop: eqn reduction}. Then there exists
some constant $C$ depending only on $M,g, n, \rho,\sigma$ such that for all $\al \geq 1$,
\[
 w_\al \leq C,\quad \mbox{on }\pa M.
\]

\end{prop}

\begin{proof} In the following, $C$ denotes some constant which may depend on $M,g, n, \rho,\sigma$ but not on $\al$ and may vary from line to line.

It suffices to prove the proposition for large $\al$, in particular, say, $\al \geq \max\{\frac{1}{|\pa M|_g},1\}$.
Let $\tilde{\rho}:=\varphi_\al^{\frac{2}{n-2\sigma}}\rho$.
Then \eqref{eqn reduction} can be rewritten as
\be\label{eqn for iteration}
\begin{cases}
\operatorname{div}_{\tilde{g}}\Big(\tilde{\rho}^{1-2\sigma}\nabla_{\tilde{g}}
w_\al\Big)=0,&\quad \mbox{in } M,\\
\dlim_{y\to \bar x}\tilde{\rho}^{1-2\sigma} \frac{\pa_{\tilde{g}}w_\al(y)}{\pa \tilde{\nu}}
\leq \xi_\al w_\al^{q-1}(\bar x),&  \quad \mbox{for }\bar x\in \pa M\setminus \{x_\al\},
\end{cases}
\ee
where the limit is taken in the sense explained in the paragraph above \eqref{E-L eqn}. In the following, we shall abuse notation a little by writing $\psi^{-1}(\mathcal{B}^+_{\delta}(0))$ as $\mathcal{B}^+_{\delta}(0)$ where $(\psi^{-1}(\mathcal{B}^+_{\delta}(0)), \psi)$ is a Fermi coordinate of $M$ at $x_\al$, and denoting $B_{\delta}(x_\al)$ as the geodesic ball on $\pa M$ centered at $x_\al$ with radius $\delta$ as before. Note that the interior of $\overline{\mathcal{B}}^+_{\delta}(0)\cap\pa M$ is $B_{\delta}(x_\al)$.

\noindent\textbf{Step 1.} We claim that there exist some constants $0<\delta_2\ll 1$, $s_0>q$ independent of $\al$ such that
\be\label{s2}
 \int_{\pa M\setminus B_{\mu_\al/\delta_2}(x_\al)} w_\al^{s_0}\,\ud s_{\tilde{g}}\leq C.
\ee

For any $\va >0$, it follows from Proposition \ref{prop: integral converge} that there exists a small $\delta_2$ such that
\be\label{s3}
\begin{split}
\int_{\pa M\setminus B_{\mu_\al/\delta_2}(x_\al)}w_\al^q\,\ud s_{\tilde{g}}& =
\int_{\pa M\setminus B_{\mu_\al/\delta_2}(x_\al)}u_\al^q\,\ud s_{g}\\&
=1-\int_{\pa' \mathcal{B}^+_{1/\delta_2}} v_\al^q\\&
\leq \va.
\end{split}
\ee

Without loss of generality, we may assume $10\mu_\al/\delta_2<\delta_0$ where $\delta_0$ is the constant such that the Fermi coordinate system centered at $x_\al$ exists in $\B^+_{\delta_0}(x_\al)$.

We choose $\eta$ to be some cutoff function satisfying
\[
\begin{split}
 &\eta(x)=1\ \mbox{if}\ |x|\geq \mu_\al/\delta_2, \quad \eta(x)=0\ \mbox{if}\ |x|\leq \mu_\al/(2\delta_2),\\&
 \mbox{and } \eta=\eta(|x|)
\mbox{ in the Fermi coordinate system centered at } x_\al.
\end{split}
\]
Multiplying \eqref{eqn for iteration} by $w_\al^k\eta^2$ for $k>1$ and
integrating by parts, we obtain
\[
 \int_{M}\tilde{\rho}^{1-2\sigma}\nabla_{\tilde{g}} w_\al \nabla_{\tilde{g}} (w^k_\al \eta^2)
 \,\ud v_{\tilde{g}}\leq \xi_\al \int_{\pa M}w_\al^{q-1+k}\eta^2\,\ud s_{\tilde{g}}.
\]
By a direct computation, we see that
\[
 \begin{split}
  &\int_{M}\tilde{\rho}^{1-2\sigma}\nabla_{\tilde{g}} w_\al \nabla_{\tilde{g}} (w^k_\al \eta^2)\,\ud v_{\tilde{g}}\\&
= \frac{4k}{(k+1)^2}\int_{M}\tilde{\rho}^{1-2\sigma}|\nabla_{\tilde{g}} (w_\al^{(k+1)/2} \eta)|^2\,\ud v_{\tilde{g}}+\frac{k-1}{(k+1)^2}
\int_{M} w^{k+1}_\al \operatorname{div}_{\tilde{g}}\Big(\tilde{\rho}^{1-2\sigma}\nabla_{\tilde{g}} \eta^2 \Big)\,\ud v_{\tilde{g}}\\&
\quad -\frac{4k}{(k+1)^2}\int_{M} \tilde{\rho}^{1-2\sigma} w^{k+1}_\al |\nabla_{\tilde{g}} \eta|^2 \,\ud v_{\tilde{g}},
 \end{split}
\]
where we have used that $\lim_{\rho\to 0} \tilde{\rho}^{1-2\sigma}  \frac{\pa_{\tilde{g}} \eta^2}{\pa \tilde{\nu}}=0$ since $\eta$ is radial.
In conclusion, we obtain
\be\label{s1}
 \begin{split}
&\int_{M}\tilde{\rho}^{1-2\sigma}|\nabla_{\tilde{g}} (w_\al^{(k+1)/2} \eta)|^2\,\ud v_{\tilde{g}}\\&
\leq -\frac{k-1}{4k}
\int_{M} w^{k+1}_\al \operatorname{div}_{\tilde{g}}\Big(\tilde{\rho}^{1-2\sigma}\nabla_{\tilde{g}} \eta^2 \Big)\,\ud v_{\tilde{g}}
 +\int_{M} \tilde{\rho}^{1-2\sigma} w^{k+1}_\al |\nabla_{\tilde{g}} \eta|^2 \,\ud v_{\tilde{g}}\\&
\quad 
+\frac{\xi_\al(k+1)^2}{4k}
\int_{\pa M} w^{q-1+k}_\al\eta^2\,\ud s_{\tilde{g}}.
 \end{split}
\ee
Since $\tilde{g}^{ij}\sim\mu_\al^2\delta^{ij}$ in
$\B^+_{2\mu_\al/\delta_2}(x_\al)\setminus \B^+_{\mu_\al/(4\delta_2)}(x_\al)$, 
we have
\[
 |\nabla_{\tilde{g}} \eta|+|\nabla_{\tilde{g}}^2 \eta|\leq C.
\]
Since $\eta$ is radial in the Fermi coordinate system, using \eqref{nb4}, \eqref{nb4.1} and \eqref{nb4.2}, we have
\[
|\operatorname{div}_{\tilde{g}}(\tilde{\rho}^{1-2\sigma}\nabla_{\tilde{g}} \eta^2)|\leq C\tilde{\rho}^{1-2\sigma}.
\]
Taking $1<k\leq q-1$ in \eqref{s1} and using Theorem \ref{thm: A1} and Theorem \ref{thm:nb2}, it follows that
\[
 \begin{split}
  &\int_{M}\tilde{\rho}^{1-2\sigma}|\nabla_{\tilde{g}} (w_\al^{(k+1)/2} \eta)|^2\,\ud v_{\tilde{g}}\\&
\leq C(k,\delta_2)+\frac{\xi_\al(k+1)^2}{4k}\int_{\pa M} w^{q-1+k}_\al\eta^2\,\ud s_{\tilde{g}}\\&
\leq C(k,\delta_2)+\frac{\xi_\al(k+1)^2}{4k}\va^{(q-2)/q}\left(\int_{\pa M} (w^{(1+k)/2}_\al\eta)^q\,\ud s_{\tilde{g}}\right)^{2/q}\\&
\leq C(k,\delta_2)+C\va^{(q-2)/q}\int_{M}\tilde{\rho}^{1-2\sigma}|\nabla_{\tilde{g}} (w_\al^{(k+1)/2} \eta)|^2\,\ud v_{\tilde{g}},
 \end{split}
\]
where we used
\be\label{s4}
\begin{split}
&\int_{M\cap (\B^+_{\mu_\al/\delta_2}\setminus \B^+_{\mu_\al/(2\delta_2)})} \tilde{\rho}^{1-2\sigma} w_\al^{k+1}\,\ud v_{\tilde{g}}\\&
\leq C(\delta_2)\int_{M\cap (\B^+_{\mu_\al/\delta_2}\setminus \B^+_{\mu_\al/(2\delta_2)})} (\frac{\rho}{\mu_\al})^{1-2\sigma} (\mu_\al^{(n-2\sigma)/2}u_\al)^{k+1}\mu_\al^{-(n+1)}\,\ud v_{g}\\&
\leq C(\delta_2)\int_{ 1/(2\delta_2)\leq |z|\leq 1/\delta_2}\rho_\al(z)^{1-2\sigma}
v_\al(z)^{k+1}\,\ud v_{g_\al}\quad \mbox{by changing variables}\\&
\leq C(k,\delta_2),
\end{split}
\ee
and $\rho_\al(z)$, $v_\al(z)$ are those in \eqref{eq 0}.

Taking $\va>0$ sufficiently small, we have
\[
 \int_{M}\tilde{\rho}^{1-2\sigma}|\nabla_{\tilde{g}} (w_\al^{(k+1)/2} \eta)|^2\,\ud v_{\tilde{g}}\leq C.
\]
The claim follows immediately from Theorem \ref{thm: A1} in the Appendix.

\medskip

\noindent\textbf{Step 2.} We shall complete the proof by Moser's iterations. Set, for $\delta=\delta_2/10$,
\[
 R_l=\mu_\al\frac{(2-2^{-(l-1)})}{\delta},\quad l=1,2,3,\dots .
\]
We choose $\eta_l$ to be some cutoff function satisfying
\[
\begin{split}
 &\eta_l(x)=1\ \mbox{if}\ |x|\geq R_{l+1}, \quad \eta_l(x)=0\ \mbox{if}\  |x|\leq R_l,\\&
 \mbox{and } \eta_l=\eta_l(|x|)
\mbox{ in the Fermi coordinate system centered at } x_\al.
\end{split}
\]
Since $\tilde{g}^{ij}\sim\mu_\al^2\delta^{ij}$ in $\B^+_{2\mu_\al/\delta_2}(x_\al)
\setminus \B^+_{\mu_\al/(4\delta_2)}(x_\al)$ and $\eta_l$ is radial in the Fermi coordinate system, we have
\[
 |\nabla_{\tilde{g}} \eta_l|\leq C2^l,\quad |\operatorname{div}_{\tilde{g}}(\tilde{\rho}^{1-2\sigma}\nabla_{\tilde{g}} \eta_l^2)|\leq C4^l\tilde{\rho}^{1-2\sigma},
 \quad \mbox{and } \lim_{\rho\to 0} \tilde{\rho}^{1-2\sigma}  \frac{\pa_{\tilde{g}} \eta_l^2}{\pa \tilde{\nu}}=0.
\]
In view of \eqref{s1}, we have
\be\label{s5}
 \begin{split}
  &\int_{M}\tilde{\rho}^{1-2\sigma}|\nabla_{\tilde{g}} (w_\al^{(k+1)/2} \eta_l)|^2\,\ud v_{\tilde{g}}\\&
\leq C4^l\int_{M\cap (\B^+_{R_{l+1}}(x_\al)\setminus \B^+_{R_l}(x_\al))} \tilde{\rho}^{1-2\sigma} w_\al^{k+1}\,\ud v_{\tilde{g}}
+\frac{C(k+1)^2}{k}\int_{\pa M\setminus B_{R_l}(x_\al)} w^{q-1+k}_\al\,\ud s_{\tilde{g}}.
 \end{split}
\ee
Set $r_0=s_0/(q-2)$, where $s_0$ is given in the step 1. It follows H\"older inequality and \eqref{s2} that
\be\label{s6}
\begin{split}
\int_{\pa M\setminus B_{R_l}(x_\al)} w^{q-1+k}_\al\,\ud s_{\tilde{g}}&=
\int_{\pa M\setminus B_{R_l}(x_\al)}w^{q-2}_\al w^{k+1}_\al\,\ud s_{\tilde{g}}\\&
\leq C\left(\int_{\pa M\setminus B_{R_l}(x_\al)} w_\al^{(k+1)r_0/(r_0-1)}\,\ud s_{\tilde{g}}\right)^{(r_0-1)/r_0}.
\end{split}
\ee
Computing as \eqref{s4}, we see that
\[
\begin{split}
 &\int_{M\cap (\B^+_{R_{l+1}}(x_\al)\setminus \B^+_{R_l}(x_\al))} \tilde{\rho}^{1-2\sigma} w_\al^{k+1}\,\ud v_{\tilde{g}}\\&
\leq C^{k+1}\int_{ 2-2^{-(l-1)}\leq \delta |z|\leq 2-2^{-l}}\rho_\al(z)^{1-2\sigma} v_\al(z)^{k+1}\,\ud v_{g_\al}\\&
\leq C^{k+1}\delta^{-1}2^{-l}\max_{\mathcal{B}^+_{2/\delta}}v_\al^{k+1},
\end{split}
\]
and
\[
 \begin{split}
  &\left(\int_{\pa M\setminus B_{R_l}(x_\al)} w_\al^{(k+1)r_0/(r_0-1)}\,\ud s_{\tilde{g}}\right)^{(r_0-1)/r_0}\\&
\geq C^{-(k+1)}\left(\int_{ 1\leq \delta |z'|\leq 2}\rho_\al(z',0)^{1-2\sigma} v_\al(z)^{(k+1)r_0/(r_0-1)}\,\ud s_{g_\al}\right)^{(r_0-1)/r_0}\\&
\geq C^{-(k+1)}\min_{\pa'\mathcal{B}^+_{2/\delta}}v_\al^{k+1}.
 \end{split}
\]
Hence, it follows from \eqref{holder estimate 3} that
\be\label{s7} \begin{split}
               &\left(\int_{M\cap (\B^+_{R_{l+1}}(x_\al)\setminus \B^+_{R_l}(x_\al))}
               \tilde{\rho}^{1-2\sigma} w_\al^{k+1}\,\ud v_{\tilde{g}}\right)^{1/(k+1)}\\&\leq
C\left(\int_{\pa M\setminus B_{R_l}(x_\al)} w_\al^{(k+1)r_0/(r_0-1)}\,\ud s_{\tilde{g}}\right)^{(r_0-1)/r_0(k+1)}
              \end{split}
\ee
It follows from Theorem \ref{thm: A1}, \eqref{s5}, \eqref{s6} and \eqref{s7} that
\be\label{s8}
\begin{split}
& \left(\int_{\pa M\setminus B_{R_{l+1}}(x_\al)} w_\al^{(k+1)q/2}\,\ud s_{\tilde{g}}\right)^{2/(k+1)q}\\&
\leq \left(C4^l+\frac{C(k+1)^2}{k}\right)^{1/(k+1)}\left(\int_{\pa M\setminus B_{R_l}(x_\al)} w_\al^{(k+1)r_0/(r_0-1)}\,\ud s_{\tilde{g}}\right)^{(r_0-1)/r_0(k+1)}.
\end{split}
\ee
Set $\chi:=\frac{r_0-1}{r_0}\cdot \frac{q}{2}=1+\frac{(s_0-q)(q-2)}{2s_0}>1$,
$q_0=\frac{2r_0}{r_0-1}$, $q_{l}=q_{l-1}\cdot  \chi=\chi^{l-1}q$ and $p_l=q_l(r_0-1)/r_0=2\chi^l$ where $l\ge 1$. Taking $k=p_l-1$ in \eqref{s8},
we obtain
\[
 \|w_\al\|_{L^{q_{l+1}}(\pa M\setminus B_{R_{l+1}})}\leq \left(C4^l+\frac{C p_l^2}{p_l-1}\right)^{1/p_l}\|w_\al\|_{L^{q_{l}}(\pa M\setminus B_{R_{l}})}.
\]
Therefore,
\[
\begin{split}
 \|w_\al\|_{L^{q_{l+1}}(\pa M\setminus B_{R_{l+1}})}&\leq \|w_\al\|_{L^{q_{1}}(\pa M\setminus B_{R_{1}})} \prod_{l=1}^\infty \left(C4^l+\frac{C p_l^2}{p_l-1}\right)^{1/p_l}\\&
\leq \|w_\al\|_{L^{p_{1}}(\pa M\setminus B_{R_{1}})} \prod_{l=1}^\infty C^{1/(2\chi^{l})}(4+\chi)^{l/(2\chi^l)}\\&
\leq C\|w_\al\|_{L^{p_{1}}(\pa M\setminus B_{R_{1}})}.
\end{split}
\]
Sending $l$ to $\infty$, we have
\be\label{s9}
\|w_\al\|_{L^\infty(\pa M\setminus B_{2\mu_\al/\delta}(x_\al))}\leq C.
\ee

By the choice of $G_\al$, $\varphi_\al(x)\geq C^{-1}\mu_\al^{-(n-2\sigma)/2}$ for $x\in B_{2\mu_\al/\delta}(x_\al)$.
Hence, for $x\in B_{2\mu_\al/\delta}(x_\al)$,
\be\label{s10}
w_{\al}(x)=\frac{u_\al(x)}{\varphi_\al(x)}\leq C \mu_\al^{(n-2\sigma)/2}u_\al(x)\leq C.
\ee
In view of \eqref{s9} and \eqref{s10}, we completed the proof of  the proposition.
\end{proof}

\begin{cor}\label{cor: up bound on the boundary} There exists a positive constant $C$ depending only on $M,g, n, \rho,\sigma$ such that
\[
 u_\al(x)\leq Cu_\al(x_\al)^{-1}\mathrm{dist}_{\pa M,g}(x,x_\al)^{2\sigma-n},\quad \mbox{for all } x\in \pa M.
\]
\end{cor}
\begin{proof}
 It follows immediately from Proposition \ref{prop: up bound on the boundary}.
\end{proof}

\section{Proofs of the main theorems}
\label{sec: proof of main thm}

Let $u_\al$ and $x_\al$ be as in
Section \ref{sec: asymp anal}. We will still use Fermi coordinates $x=(x_1,\cdots, x_{n+1})$ centered at $x_\al$.
In this coordinate system, 
\[
\sum_{1\leq i,j\leq n+1}g_{ij}(x)\ud x_i\ud x_j =\ud x_{n+1}^2+\sum_{1\leq i,j\leq n}g_{ij}(x)\ud x_i\ud x_j, \quad \mbox{for }|x|\leq \delta_0,
\]
where $\delta_0>0$ is independent of $\al$.  
Then we have
\be\label{p1}
\begin{cases}
 \mathrm{div}_{g}\Big(\rho(x)^{1-2\sigma}\nabla_{g}u_\al(x)\Big)=0,&\quad \mbox{in }\mathcal{B}_{\delta_0}^{+},\\
-\dlim_{x_{n+1}\to 0^+}\rho(x)^{1-2\sigma}\frac{\pa u_\al}{\pa x_{n+1}}=\xi_\al u_\al^{q-1}(x',0)-\al u_\al(x',0),&\quad \mbox{on }\pa' \mathcal{B}_{\delta_0}^{+}.
\end{cases}
\ee

\begin{prop}\label{prop: interior upbound}
There exists a positive constant $C$ independent of $\al$ such that
\[
u_{\al}(x)\leq Cu_\al(0)^{-1}|x|^{2\sigma-n},\quad \mathcal{B}^+_{10 \al^{-1/2\sigma}}(0).
\]
\end{prop}

\begin{proof}
By Corollary \ref{cor: up bound on the boundary},
\be\label{p2}
 u_{\al}(x',0)\leq Cu_\al(0)^{-1}|x'|^{2\sigma-n},\quad |x'|\leq \delta_0.
\ee
Let $r:=|\overline{x}|<10 \al^{-1/2\sigma}$, $\phi_\al(x)=r^{\frac{n-2\sigma}{2}}u_\al(r x)$. Then $\phi_\al$ satisfies
\be\label{p3}
\begin{cases}
 \mathrm{div}_{\hat{g}}\Big(\hat{\rho}(x)^{1-2\sigma}\nabla_{\hat{g}}\phi_\al(x)\Big)=0,&\quad \mbox{in }\mathcal{B}_{\delta_0/r}^{+},\\
-\dlim_{x_{n+1}\to 0^+}\hat{\rho}(x)^{1-2\sigma}\frac{\pa \phi_\al}{\pa x_{n+1}}=\xi_\al \phi_\al^{q-1}(x',0)-\al r^{2\sigma} \phi_\al(x',0),
&\quad \mbox{on }\pa' \mathcal{B}_{\delta_0/r}^{+},
\end{cases}
\ee
where $\hat{\rho}(x)=\rho(r x)/r$, $\hat{g}(x)=g_{ij}(r x)\ud x_i\ud x_j$. Since $x_\al=0$ is a maximum point of $u_\al$, it follows from
\eqref{p2} that
\be\label{p5}
\phi_{\al}(x',0)=r^{\frac{n-2\sigma}{2}}u_\al(r x',0)\leq C r^{\frac{n-2\sigma}{2}} (r|x'|)^{-\frac{n-2\sigma}{2}}\le C,\quad \frac{1}{2}< |x'|< 2.
\ee
Applying the Harnack inequality in \cite{CaS} or \cite{TX} and standard Harnack inequality for uniformly elliptic equations to $\phi_\al$ in $\{x:\frac{1}{2}<|x|<2,x_{n+1}>0\}$, we conclude that
\[
 \max_{\mathcal{B}_{3/2}^+\setminus \mathcal{B}_{3/4}^+}\phi_\al \leq C \min_{\mathcal{B}_{3/2}^+\setminus \mathcal{B}_{3/4}^+}\phi_\al.
\]
Hence, by \eqref{p2}
\[
 u_\al(\overline{x})\leq C u(\tilde x', 0)\leq Cu_\al(0)^{-1}|\overline{x}|^{2\sigma-n},
\]
where $|\tilde x'|=|\bar x|$.
By the arbitrary choice of $\overline{x}$,
the proposition follows immediately.
\end{proof}

Let $\mu_\al=u_\al(0)^{-\frac{2}{n-2\sigma}}$, $R_\al=(\al^{1/2\sigma}\mu_\al)^{-1}$,  $g_\al
=g_{ij}(\mu_\al x)\ud x_i\ud x_j$ and $\rho_\al(x)=\frac{\rho(\mu_\al x)}{\mu_\al}$ in
$\mathcal{B}^+_{10 R_\al}$. Set $v_\al(x)=\mu_\al^{\frac{n-2\sigma}{2}}u_\al(\mu_\al x)$
for $x\in \mathcal{B}^+_{10 R_\al}$. It follows that
\be\label{eq 00}
\begin{cases}
 \operatorname{div}_{g_\al}\Big(\rho_\al^{1-2\sigma}\nabla_{g_\al} v_\al\Big)=0,&\quad \mbox{in }\mathcal{B}_{10 R_\al}^+\\
\lim_{x_{n+1}\to 0}\rho_\al^{1-2\sigma}\frac{\pa_{g_\al} v_{\al}}{\pa \nu}=\xi_\al v_\al^{q-1}-\al \mu_\al^{2\sigma} v_\al,& \quad \mbox{on }\pa' \mathcal{B}_{10 R_\al}^+=B_{10 R_\al}\\
v_\al(0)=1,\quad  0< v_\al \leq 1.
\end{cases}
\ee
By Proposition \ref{prop: interior upbound},
\be\label{iup1}
v_\al(x)\leq \frac{C}{1+|x|^{n-2\sigma}},\quad x \in \overline{\mathcal{B}}^+_{10 R_\al}.
\ee

\begin{prop} \label{prop: iupb}
For all $\al\geq 1$, $x\in \mathcal{B}^+_{R_\al}(0)$, we have
\[
\begin{split}
 |\nabla_{x'}v_\al(x',x_{n+1})|&\leq \frac{C}{1+|x|^{n+1-2\sigma}},\\
|\nabla^2_{x'}v_\al(x',x_{n+1})|&\leq \frac{C}{1+|x|^{n+2-2\sigma}},\\
|\pa_{n+1}v_\al(x',x_{n+1})|&\leq \frac{C x_{n+1}^{2\sigma-1}}{1+|x|^{n}}.
\end{split}
\]
\end{prop}

\begin{proof}
Given Theorem \ref{thm: schauder} and Proposition \ref{thm: schauder-1}, the proofs follow from \eqref{iup1} and standard rescaling arguments (see, e.g., Proposition 3.1 of \cite{LZ2}).
\end{proof}

\begin{proof}[Proof of Theorem \ref{thm: main thm A}] We complete
the proof of Theorem \ref{thm: main thm A} by checking balance via a Pohozaev type inequality.

It follows from direct computations that
 \be\label{dp1}
\begin{split}
&2\mathrm{div}(x_{n+1}^{1-2\sigma}\nabla v_\al)(\nabla v_\al \cdot x)\\&
=\mathrm{div}\big(2x_{n+1}^{1-2\sigma}(\nabla v_\al \cdot x)\nabla v_\al-x_{n+1}^{1-2\sigma}
|\nabla v_\al|^2x\big)+(n-2\sigma)x_{n+1}^{1-2\sigma}|\nabla v_\al|^2.
\end{split}
\ee
Integrating both sides of \eqref{dp1} over $\mathcal{B}^+_{R_\al}$, we have
\be\label{dp2}
 \begin{split}
  &\int_{\mathcal{B}^+_{R_\al}}\mathrm{div}(x_{n+1}^{1-2\sigma}\nabla v_\al)(\nabla v_\al \cdot x)\,\ud x-\frac{n-2\sigma}{2}
\int_{\mathcal{B}^+_{R_\al}}x_{n+1}^{1-2\sigma}|\nabla v_\al|^2\,\ud x\\&
=\frac{1}{2}\int_{\mathcal{B}^+_{R_\al}}\mathrm{div}\big(2x_{n+1}^{1-2\sigma}(\nabla v_\al \cdot x)
\nabla v_\al-x_{n+1}^{1-2\sigma}|\nabla v_\al|^2x\big)\,\ud x.
 \end{split}
\ee
Integrating by parts, we obtain
\[
 \begin{split}
  &\frac{1}{2}\int_{\mathcal{B}^+_{R_\al}}\mathrm{div}\big(2x_{n+1}^{1-2\sigma}(\nabla v_\al \cdot x)
  \nabla v_\al-x_{n+1}^{1-2\sigma}|\nabla v_\al|^2x\big)\,\ud x\\&
=-\int_{\pa' \mathcal{B}^+_{R_\al}}\Big(\sum_{i=1}^{n}x_i\frac{\pa v_\al}{\pa x_i}\Big)
\frac{\pa v_\al}{\pa x^\sigma_{n+1}}\,\ud x'
 +\int_{\pa'' \mathcal{B}^+_{R_\al}}|x|x_{n+1}^{1-2\sigma}\Big(\big(\frac{\pa v_\al}{\pa \nu}\big)^2
 -\frac{1}{2}|\nabla v_\al|^2\Big)\,\ud S\\&
=-\int_{\pa' \mathcal{B}^+_{R_\al}}\Big(\sum_{i=1}^{n}x_i\frac{\pa v_\al}{\pa x_i}\Big)
\frac{\pa v_\al}{\pa x^\sigma_{n+1}}\,\ud x'
 +\int_{\pa'' \mathcal{B}^+_{R_\al}}\frac{|x|}{2}x_{n+1}^{1-2\sigma}
 \Big(\big(\frac{\pa v_\al}{\pa \nu}\big)^2-|\pa_{\mathrm{tan}} v_\al|^2\Big)\,\ud S,
 \end{split}
\]
where $\frac{\pa v_\al}{\pa x^\sigma_{n+1}}:=\dlim_{x_{n+1}\to 0^+}x_{n+1}^{1-2\sigma}
\frac{\pa v_\al}{\pa x_{n+1}}$ and $\pa_{\mathrm{tan}}$ denotes the tangential
differentiation on $\pa''\mathcal{B}^+_{R_\al}$.
On the other hand,
\[
\begin{split}
 \int_{\mathcal{B}^+_{R_\al}}x_{n+1}^{1-2\sigma}|\nabla v_\al|^2\,\ud x
=&-\int_{\mathcal{B}^+_{R_\al}}\mathrm{div}(x_{n+1}^{1-2\sigma}\nabla v_\al) v_\al\,\ud x\\&
-\int_{\pa' \mathcal{B}^+_{R_\al}}v_\al\frac{\pa v_\al}{\pa x^\sigma_{n+1}}\,\ud x'
+\int_{\pa'' \mathcal{B}^+_{R_\al}}x_{n+1}^{1-2\sigma}v_\al \frac{\pa v_\al}{\pa \nu}\,\ud S.
\end{split}
\]
In summary, we obtain
\be\label{dp3}
 \begin{split}
  &\int_{\mathcal{B}^+_{R_\al}}\mathrm{div}(x_{n+1}^{1-2\sigma}\nabla v_\al)(\nabla v_\al \cdot x)\,\ud x+\frac{n-2\sigma}{2}
\int_{\mathcal{B}^+_{R_\al}}\mathrm{div}(x_{n+1}^{1-2\sigma}\nabla v_\al) v_\al\,\ud x\\&
=B'(R_\al, v_\al)+B''(R_\al, v_\al),
 \end{split}
\ee
where
\[
\begin{split}
 B'(R_\al, v_\al)&=-\frac{1}{2} \int_{\pa' \mathcal{B}^+_{R_\al}}2\Big(\sum_{i=1}^{n}x_i\frac{\pa v_\al}{\pa x_i}\Big)
 \frac{\pa v_\al}{\pa x^\sigma_{n+1}}
+(n-2\sigma)v_\al\frac{\pa v_\al}{\pa x^\sigma_{n+1}}\,\ud x',\\
 B''(R_\al, v_\al)&=\frac{1}{2}\int_{\pa'' \mathcal{B}^+_{R_\al}}|x|x_{n+1}^{1-2\sigma}\Big(\big(\frac{\pa v_\al}{\pa \nu}\big)^2-|\pa_{\mathrm{tan}} v_\al|^2\Big)
+(n-2\sigma)x_{n+1}^{1-2\sigma}v_\al \frac{\pa v_\al}{\pa \nu}\,\ud S.
\end{split}
\]
Note that
\be\label{eq:g-e}
 \begin{split}
&  \mathrm{div}_{g_\al}(\rho_\al^{1-2\sigma}\nabla_{g_\al} v_\al)\\&
=g_\al^{ij}\frac{\pa v_\al}{\pa x_i}\frac{\pa \rho_\al^{1-2\sigma}}{\pa x_j}+\rho_\al^{1-2\sigma}g_\al^{ij}(\frac{\pa^2 v_\al}{\pa x_i\pa x_j}-\Gamma^k_{ij}\frac{\pa v_\al}{\pa x_k})\\&
=\mathrm{div}(x_{n+1}^{1-2\sigma}\nabla v_\al)+\sum_{1\leq i,j\leq n}g_\al^{ij}\frac{\pa v_\al}{\pa x_i}\frac{\pa \rho_\al^{1-2\sigma}}{\pa x_j}+
\Big(\frac{\pa \rho_\al^{1-2\sigma}}{\pa x_{n+1}}-\frac{\pa  x_{n+1}^{1-2\sigma}}{\pa x_{n+1}}\Big)\frac{\pa v_\al}{\pa x_{n+1}}\\&
\quad +\rho_\al^{1-2\sigma}(g_\al^{ij}-\delta^{ij})\frac{\pa^2 v_\al}{\pa x_i\pa x_j}+(\rho_\al^{1-2\sigma}-x_{n+1}^{1-2\sigma})\Delta v_\al -\rho_\al^{1-2\sigma}g_\al^{ij}
\Gamma^k_{ij}\frac{\pa v_\al}{\pa x_k},
 \end{split}
\ee
where $\Gamma^k_{ij}$ is the Christoffel symbol of  $g_\al$.
It is easy to see that
\begin{subequations}\label{dp4}
\begin{align}
|h_\al^{ij}(x)-\delta^{ij}|&\leq C\mu_\al |x|,\label{dp4-a}\\
|\Gamma^k_{ij}|&\leq C \mu_\al,\\
|\rho_\al(x)^{1-2\sigma}-x_{n+1}^{1-2\sigma}|&\leq C\mu_\al x_{n+1}^{2-2\sigma},\\
\left|\frac{\pa \rho_\al(x)^{1-2\sigma}}{\pa x_i}\right|&\leq C\mu_\al x_{n+1}^{1-2\sigma}\quad \mbox{for}\quad i<n+1,\\
\left| \frac{\pa \rho_\al(x)^{1-2\sigma}}{\pa x_{n+1}}-\frac{\pa x_{n+1}^{1-2\sigma}}{\pa x_{n+1}}\right |&\leq C\mu_\al x_{n+1}^{1-2\sigma}\label{dp4-e}.
\end{align}
\end{subequations}
Indeed,
\[
\begin{split}
 |\rho_\al(x)^{1-2\sigma}-x_{n+1}^{1-2\sigma}|&=x_{n+1}^{1-2\sigma}\left|\Big(\frac{\rho(\mu_\al x)}{\mu_\al x_{n+1}}\Big)^{1-2\sigma}-1\right|\\&
=x_{n+1}^{1-2\sigma}\left|\Big(\frac{\mu_\al x_{n+1}+O(\mu_\al x_{n+1})^2}{\mu_\al x_{n+1}}\Big)^{1-2\sigma}-1\right|\\&
\leq C \mu_\al x_{n+1}^{2-2\sigma},
\end{split}
\]
and
\[
\begin{split}
\frac{\pa \rho_\al(x)^{1-2\sigma}}{\pa x_i}&=(1-2\sigma)\rho_\al(x)^{-2\sigma} \Big(\frac{\pa \rho_\al(x)}{\pa x_i}-
\frac{\pa \rho_\al(x',0)}{\pa x_i}\Big)\\&
 =O(1)\mu_\al \rho_\al^{1-2\sigma}\\&
\leq C\mu_\al  x_{n+1}^{1-2\sigma}.
\end{split}
\]
It follows from \eqref{eq 00}, \eqref{dp3}, \eqref{eq:g-e} and \eqref{dp4-a}-\eqref{dp4-e} that
\be\label{dp5}
 \begin{split}
&B'(R_\al, v_\al)+B''(R_\al, v_\al)\\&
\leq C\mu_\al \int_{\mathcal{B}^+_{R_\al}} x_{n+1}^{1-2\sigma}(v_\al+|\nabla v_\al\cdot x|)(|\nabla v_\al|+|x||\nabla_{x'}^2 v_\al|
+x_{n+1}|\Delta v_\al|) \,\ud x.
 \end{split}
\ee
Since $\dlim_{x_{n+1}\to 0}\rho_\al^{1-2\sigma}\frac{\pa_{g_\al} v_\al}{\pa \nu}=-\frac{\pa  v_\al}{\pa x_{n+1}^\sigma}$ on $\pa' \mathcal{B}^+_{R_\al}$,
\[
\begin{split}
  B'(R_\al, v_\al)&= \int_{\pa' \mathcal{B}^+_{R_\al}} \Big(\sum_{i=1}^{n}x_i\frac{\pa v_\al}{\pa x_i}\Big)(\xi_\al v_\al^{q-1}-\al \mu_\al^{2\sigma}v_\al)
+\frac{(n-2\sigma)}{2}(\xi_\al v_\al^{q}-\al \mu_\al^{2\sigma}v_\al^2)\,\ud x'\\&
=\sigma \al \mu_\al^{2\sigma} \int_{\pa' \mathcal{B}^+_{R_\al}} v_\al^2\,\ud x'+\int_{\pa B_{R_\al}}(\frac{\xi_\al }{q}v_\al^q-\frac{\al \mu_\al^{2\sigma}}{2}v_\al^2)R_\al\,\ud S,
\end{split}
\]
where integrations by parts were used in the second equality.
Clearly,
\[
 B''(R_\al, v_\al)=O\left(\int_{\pa''\mathcal{B}_{R_\al}^+}x_{n+1}^{1-2\sigma}(|x||\nabla v_\al|^2+v_\al|\nabla v_\al|)\,\ud S \right).
\]
Therefore, we obtain
\be\label{dp6}
\begin{split}
 \al \mu_\al^{2\sigma} &\int_{\pa' \mathcal{B}^+_{R_\al}} v_\al^2\,\ud x' \\&
\leq C\mu_\al \int_{\mathcal{B}^+_{R_\al}} x_{n+1}^{1-2\sigma}(v_\al+|\nabla v_\al\cdot x|)(|\nabla v_\al|+|x||\nabla_{x'}^2 v_\al|
+x_{n+1}|\Delta v_\al|) \,\ud x\\&
\quad +C \int_{\pa''\mathcal{B}_{R_\al}^+}x_{n+1}^{1-2\sigma}(|x||\nabla v_\al|^2+v_\al|\nabla v_\al|)\,\ud S+ C \int_{\pa B_{R_\al}} \al \mu_\al^{2\sigma}v_\al^2 R_\al\,\ud S.
\end{split}
\ee
Since $\mathrm{div}_{g_\al}(\rho_\al^{1-2\sigma}\nabla_{g_\al} v_\al)=0$ and $g_\al^{i,n+1}=0$ for $i<n+1$, 
\be\label{dp7}
 |\pa^2_{n+1}v_\al(x',x_{n+1})|\leq C(\mu_\al |\nabla v_\al|+|\pa_{x+1}v_\al|x_{n+1}^{-1}+|\nabla^2_{x'}v_\al|). 
\ee
It follows from \eqref{dp6}, \eqref{dp7} and  Proposition \ref{prop: iupb} that
\[
 \begin{split}
 \al \mu_\al^{2\sigma} &\int_{\pa' \mathcal{B}^+_{R_\al}} v_\al^2\,\ud x' \\&
\leq C\mu_\al \int_{\mathcal{B}^+_{R_\al}} x_{n+1}^{1-2\sigma}(v_\al+|\nabla v_\al\cdot x|)(|\nabla v_\al|+|x||\nabla_{x'}^2 v_\al|
) \,\ud x\\&
\quad +C \int_{\pa''\mathcal{B}_{R_\al}^+}x_{n+1}^{1-2\sigma}(\frac{1}{R_\al^{2n+1-4\sigma}}+\frac{x_{n+1}^{2\sigma-1}}{R_\al^{2n-2\sigma}}
+\frac{x_{n+1}^{4\sigma-2}}{R_\al^{2n-1}})\,\ud S+ C \frac{\al \mu_\al^{2\sigma}}{R_\al^{n-4\sigma}}\\&
\leq  C\mu_\al \int_{\mathcal{B}^+_{R_\al}}(\frac{x_{n+1}^{1-2\sigma}}{(1+|x|)^{2n+1-4\sigma}}+\frac{1}{(1+|x|)^{2n-2\sigma}})\,\ud x\\&
\quad + CR_{\al}^{2\sigma-n}\int_{\pa''\mathcal{B}_1}(y_{n+1}^{1-2\sigma}+1+y_{n+1}^{2\sigma-1})\,\ud S + C \frac{\al \mu_\al^{2\sigma}}{R_\al^{n-4\sigma}}\\&
\leq \begin{cases}
      C\mu_\al \ln R_\al +C(\al\mu_\al^{2\sigma} )^{\frac{n-2\sigma}{2\sigma}}+C\al \mu_\al^{2\sigma}R_\al^{4\sigma-n},\quad n=2\sigma+1\\
    C\mu_\al  +C(\al\mu_\al^{2\sigma} )^{\frac{n-2\sigma}{2\sigma}}+C\al \mu_\al^{2\sigma}R_\al^{4\sigma-n}, \quad n>2\sigma+1.
     \end{cases}
\end{split}
\]
For $\sigma=1/2$ and $n=2$, Theorem \ref{thm: main thm A} was proved in \cite{LZ2}.  Hence, we may assume that $n>2\sigma+1$. Since $\sigma\in (0,1/2]$, $n>2\sigma+1\ge 4\sigma$. Therefore,
\[
0< \frac{1}{C}\leq\int_{\pa' \mathcal{B}^+_{R_\al}} v_\al^2\,\ud x' \to 0,\quad \mbox{as }\al \to \infty
\]
which is a contradiction. 
\end{proof}

\begin{proof}[Proof of Theorem \ref{thm: main thm B}]
Since $\pa M$ is totally geodesic, Lemma \ref{taylor expansion of fermi} implies that
\begin{subequations}\label{dp8-1}
\begin{align}
|h_\al^{ij}(x)-\delta^{ij}|&\leq C\mu_\al^2 |x|^2,\\
|\Gamma^k_{ij}|&\leq C \mu_\al^2 |x|.
\end{align}
\end{subequations}
Since $\rho=d(x)+O(d(x)^3)$, it follows that
\begin{subequations}\label{dp8-2}
\begin{align}
|\rho_\al(x)^{1-2\sigma}-x_{n+1}^{1-2\sigma}|&\leq C\mu_\al^2 x_{n+1}^{3-2\sigma},\\
\left|\frac{\pa \rho_\al(x)^{1-2\sigma}}{\pa x_i}\right|&\leq C\mu_\al^2 x_{n+1}^{2-2\sigma},\quad i<n+1,\\
\left| \frac{\pa \rho_\al(x)^{1-2\sigma}}{\pa x_{n+1}}-\frac{\pa x_{n+1}^{1-2\sigma}}{\pa x_{n+1}}\right |&\leq C\mu_\al^2 x_{n+1}^{2-2\sigma}.
\end{align}
\end{subequations}
Similar to \eqref{dp6}, we have 
\be\label{dp9}
\begin{split}
 \al \mu_\al^{2\sigma} &\int_{\pa' \mathcal{B}^+_{R_\al}} v_\al^2\,\ud x' \\&
\leq C\mu_\al^2 \int_{\mathcal{B}^+_{R_\al}} x_{n+1}^{1-2\sigma}(v_\al+|\nabla v_\al\cdot x|)(|x| |\nabla v_\al|+|x|^2|\nabla_{x'}^2 v_\al|
+x^2_{n+1}|\Delta v_\al|) \,\ud x\\&
\quad +C \int_{\pa''\mathcal{B}_{R_\al}^+}x_{n+1}^{1-2\sigma}(|x||\nabla v_\al|^2+v_\al|\nabla v_\al|)\,\ud S+ C \int_{\pa B_{R_\al}} \al \mu_\al^{2\sigma}v_\al^2 R_\al\,\ud S.
\end{split}
\ee
It follows from \eqref{dp7}, \eqref{dp9} and Proposition \ref{prop: iupb} that
\[
 \begin{split}
 \al \mu_\al^{2\sigma} &\int_{\pa' \mathcal{B}^+_{R_\al}} v_\al^2\,\ud x' \\&
 \leq C\mu_\al^2 \int_{\mathcal{B}^+_{R_\al}} x_{n+1}^{1-2\sigma}(v_\al+|\nabla v_\al\cdot x|)(|x| |\nabla v_\al|+|x|^2|\nabla_{x'}^2 v_\al|
) \,\ud x\\&
\quad +C \int_{\pa''\mathcal{B}_{R_\al}^+}x_{n+1}^{1-2\sigma}(|x||\nabla v_\al|^2+v_\al|\nabla v_\al|)\,\ud S+ C \int_{\pa B_{R_\al}} \al \mu_\al^{2\sigma}v_\al^2 R_\al\,\ud S\\ &
\leq  C\mu_\al^2 \int_{\mathcal{B}^+_{R_\al}}\frac{x_{n+1}^{1-2\sigma}}{(1+|x|)^{2n-4\sigma}}\,\ud x+C(\al\mu_\al^{2\sigma} )^{\frac{n-2\sigma}{2\sigma}}+C\al \mu_\al^{2\sigma}R_\al^{4\sigma-n}\\&
\leq C\mu_\al^2+C(\al\mu_\al^{2\sigma} )^{\frac{n-2\sigma}{2\sigma}}+C\al \mu_\al^{2\sigma}R_\al^{4\sigma-n},
\end{split}
\]
provided $n>2+2\sigma$ (i.e., $n\geq 4$).
Therefore,
\[
0< \frac{1}{C}\leq\int_{\pa' \mathcal{B}^+_{R_\al}} v_\al^2\,\ud x' \to 0\quad \mbox{as }\al \to \infty,
\]
which is a contradiction.
\end{proof}

\appendix

\section{Appendix}

\subsection{A trace inequality}

Let $(M,g)$ be a smooth, compact Riemannian manifold of dimension $n+1$ $(n\geq 2)$ with boundary.

\begin{lem}\label{lem:A1}
For $n\ge 2$, there exists some positive constant $C=C(n,\sigma)$ such that for all $u\in H^1(x_{n+1}^{1-2\sigma}, \mathcal{B}_1^+)$, $u\equiv 0$ in an open neighborhood of $x=0$, we have 
 \[
 \left(\int_{\pa' \mathcal{B}_1^+}\frac{|u(x',0)|^q}{|x'|^{2n}}\,\ud x'\right)^{2/q}\leq C\int_{\mathcal{B}_1^+}\frac{x_{n+1}^{1-2\sigma}
|\nabla u|^2}{|x|^{2n-4\sigma}}\,\ud x.
\]
\end{lem}

\begin{proof}
By the assumption of $u$, there exists a positive constant $\mu=\mu(u)>0$ such that $u\equiv 0$ for $|x|<\mu $ with $ x_{n+1}>0$.
Consider
\[
 v(y)=u\left(\frac{y}{|y|^2}\right), \quad |y|>1, y_{n+1}>0.
\]
It is easy to see that
\[
 v(y)\equiv 0,\quad \mbox{for all }|y|>1/\mu,\ y_{n+1}>0,
\]
and for some $C(n)>0$,
\[
 \int_{\pa' \mathcal{B}_1^+}\frac{|u(x',0)|^q}{|x'|^{2n}}\,\ud x'=C(n)\int_{|y'|\geq 1 }|v(y',0)|^q\,\ud y',
\]
and
\[
 \int_{\mathcal{B}_1^+}\frac{x_{n+1}^{1-2\sigma}
|\nabla u|^2}{|x|^{2n-4\sigma}}\,\ud x=C(n)\int_{|y|\geq 1, y_{n+1}>0}
y_{n+1}^{1-2\sigma}|\nabla v(y)|^2\,\ud y.
\]
By some appropriate extension of $v$ to $|y|<1$, it follows from \eqref{sharp trace} that
\[
 \int_{|y'|\geq 1, }|v(y',0)|^q\,\ud y'\leq C(n,\sigma) \int_{|y|\geq 1, y_{n+1}>0}
y_{n+1}^{1-2\sigma}|\nabla v(y)|^2\,\ud y.
\]
The proof is completed.
\end{proof}

\begin{lem}\label{lem:A1-1}
For $\delta>0$, there exists $C=C(M,g,n,\sigma, \delta, \rho)>0$ such that for all $x_0\in\pa M$, $u\in H^1(\rho^{1-2\sigma},M\setminus \mathcal{B}_{\delta/2}(x_0))$, we have
\be\label{A2}
\begin{split}
& \left(\int_{\pa M\setminus B_\delta(x_0)} |u(x)|^q\right)^{2/q}+
\int_{ M\setminus \mathcal{B}^+_{\delta}(x_0)} \rho^{1-2\sigma}|u(x)|^2\\&
\leq C\left\{\int_{ M\setminus \mathcal{B}^+_{\delta/2}(x_0)} \rho^{1-2\sigma}|\nabla_g u|^2+
 \int_{\pa M\cap(B_\delta(x_0)\setminus \overline{B}_{\delta/2}(x_0))} |u(x)|^2\right\}.
\end{split}
\ee
\end{lem}

\begin{proof}

We prove \eqref{A2} by contradiction.
Suppose the contrary of \eqref{A2}  that for some $\delta>0$, there exists a
sequence of points $\{x_i\}\in \pa M$, $\{u_i\}\in H^1(\rho^{1-2\sigma},  M\setminus \mathcal{B}^+_{\delta/2}(x_i))$
satisfying
\be\label{A3}
\left(\int_{\pa M\setminus B_\delta(x_i)} |u_i(x)|^q\right)^{2/q}+
\int_{ M\setminus \mathcal{B}^+_{\delta}(x_i)} \rho^{1-2\sigma}|u_i(x)|^2=1,
\ee
but
\be \label{A4}
 \int_{ M\setminus \mathcal{B}^+_{\delta/2}(x_i)} \rho^{1-2\sigma}|\nabla_g u_i|^2+
 \int_{\pa M\cap(B_\delta(x_i)\setminus \overline{B}_{\delta/2}(x_i))} |u_i(x)|^2\leq \frac{1}{i}.
\ee
After passing to some subsequence, $\{u_i\}$ converges weakly to $u$ in
$H^1(\rho^{1-2\sigma},  M\setminus \mathcal{B}^+_\delta(x_i))$. By \eqref{A4}, $u\equiv 0$. It follows from a compact Sobolev embedding in Proposition \ref{sharp sobolev embedding} that
\[
 \int_{ M\setminus \B^+_{\delta}(x_i)} \rho^{1-2\sigma}|u_i(x)|^2\to 0.
\]
By a trace embedding in Proposition \ref{prop: weak ineq}, we also conclude that
\[
 \left(\int_{\pa M\setminus B_\delta(x_i)} |u(x)|^q\right)^{2/q}\to 0.
\]
Therefore, we reach a contradiction to \eqref{A3}.
\end{proof}

\begin{thm}\label{thm: A1}
 There exists some constant $C=C(M,g,\rho, n, \sigma)$ such that for all
 $x_0\in \pa M$, $\mu>0$, $u\in H^1(\rho^{1-2\sigma}, M)$, $u \equiv 0$ in $\{x\in M: \mathrm{dist}(x,x_0)<\mu$\}, we have
\[
 \left(\int_{\pa M}\frac{|u(x)|^q}{\mathrm{dist}(x,x_0)^{2n}}\,\ud s_g\right)^{2/q}\leq C\int_{M}\frac{\rho^{1-2\sigma}
 |\nabla_g u|^2}{\mathrm{dist}(x,x_0)^{2n-4\sigma}}\,\ud v_g.
\]
\end{thm}

\begin{proof}
The theorem follows clearly from Lemma \ref{lem:A1} and Lemma \ref{lem:A1-1}.
\end{proof}

\subsection{Regularity results for degenerate elliptic equations}\label{sec of weak solution}

Suppose that $a^{ij}(x)$, $1\leq i,j\leq n+1$, is a smooth positive definite matrix-valued in $\mathcal{B}^+_2$ and
there exists a positive constant $\Lda\geq 1$ such that
\[
\frac{1}{\Lda}|\xi|^2\leq a^{ij}\xi_i\xi_j\leq \Lda|\xi|^2,\quad \forall \xi \in \R^{n+1}
\]
Suppose also that
\[
 a^{i,n+1}=a^{n+1,i}=0 \ \mbox{for}\  i<n+1.
\]
Consider
\be\label{app: main eq}
\begin{cases}
 \frac{\pa}{\pa x_i}\big(x_{n+1}^{1-2\sigma}a^{ij}(x)\frac{\pa}{\pa x_j}u(x)\big)=0,&\quad \mbox{in } \mathcal{B}^+_2,\\
-\dlim_{x_{n+1}\to 0^+}x_{n+1}^{1-2\sigma}a^{n+1,n+1}\frac{\pa u(x)}{\pa x_{n+1}}=b(x')u+f(x'),& \quad \mbox{on } \pa'\mathcal{B}^+_2.
\end{cases}
\ee
We say $u\in H^{1}(x_{n+1}^{1-2\sigma}, \mathcal{B}^+_2)$ is a weak solution of \eqref{app: main eq} if
\[
\int_{\mathcal{B}^+_2}x_{n+1}^{1-2\sigma}a^{ij}(x)\frac{\pa u}{\pa x_j}\frac{\pa \varphi}{\pa x_i}=\int_{\pa' \mathcal{B}^+_2}b(x')u(x',0)\varphi(x',0)+f(x')\varphi(x',0)
\]
for all $\varphi\in C_c^{\infty}(\mathcal{B}^+_2\cup \pa'\mathcal{B}^+_2)$.
\begin{thm}\label{thm: L-infty}
Suppose that $b, f \in L^p(B_2)$ for some $p >\frac{n}{2\sigma}$.
Let $u\in H^{1}(x_{n+1}^{1-2\sigma}, \mathcal{B}^+_2)$ be a weak solution of \eqref{app: main eq}.
Then there exist constants $\gamma\in (0,1)$, $C>0$ depending only
on $n,\sigma, \Lambda, p, \|b\|_{L^p(B_2)}$ such that $u\in C^\gamma(\mathcal{B}^+_1)$ and
\[
 \|u\|_{C^\gamma(\mathcal{B}^+_1)}\leq C(\|u\|_{L^1(x_{n+1}^{1-2\sigma}, \mathcal{B}^+_2)}+\|f\|_{L^p(B_2)}).
\]
\end{thm}
\begin{proof}
It follows from a modification of the proof of Proposition 2.4 in \cite{JLX1}, which uses standard Moser iteration techniques.
\end{proof}

\begin{thm}\label{thm: schauder}
Suppose that $b, f \in C^\beta(B_2)$ for some $0<\beta\notin\mathbb{N}$. Let
$u\in H^{1}(x_{n+1}^{1-2\sigma}, \mathcal{B}^+_2)$ be a weak solution of \eqref{app: main eq}.
Suppose that $2\sigma+\beta$ is not an integer. Then $x_{n+1}^{1-2\sigma}\frac{\pa u(x)}{\pa x_{n+1}}\in C(\overline{\mathcal{B}^+_1})$,
and $u(\cdot, 0)\in C^{2\sigma+\beta}(B_1)$. Moreover,
\[
 \left|x_{n+1}^{1-2\sigma}\frac{\pa u(x)}{\pa x_{n+1}}\right|_{C(\overline{\mathcal{B}^+_1})}+
 \|u(\cdot, 0)\|_{C^{2\sigma+\beta}(B_1)}\leq C(\|u\|_{L^2(x_{n+1}^{1-2\sigma}, \mathcal{B}^+_2)}+\|f\|_{C^\beta(B_2)}),
\]
where $C>0$ depending only on $n,\sigma, \Lambda, \beta, \|b\|_{C^\beta(B_2)}$.
\end{thm}
\begin{proof}
It follows from modifications of the proofs of Theorem 2.3 and Lemma 2.3 in \cite{JLX1}.
\end{proof}

\begin{prop}\label{thm: schauder-1}
Let $b, f \in C^k(B_2)$, $u\in H^{1}(x_{n+1}^{1-2\sigma}, \mathcal{B}^+_2)$ be a weak solution of \eqref{app: main eq}, where $k$ is a positive integer.
Then we have
\[
 \sum_{j=1}^k\|\nabla_{x'}^j u\|_{L^{\infty}(\mathcal B_1^+)}\leq C(\|u\|_{L^2(x_{n+1}^{1-2\sigma}, \mathcal{B}^+_2)}+\|f\|_{C^k(B_2)}),
\]
where $C>0$ depending only on $n,\sigma, \Lambda, \beta, \|b\|_{C^k(B_2)}$.
\end{prop}
\begin{proof}
It follows from a modification of the proof of Proposition 2.5 in \cite{JLX1}.
\end{proof}

\subsection{Degenerate elliptic equations with conormal boundary conditions involving measures} 

We start with some Sobolev embeddings. For every $p\in [1,+\infty)$, we define $W^{1,p}(\rho^{1-2\sigma}, M)$ as the closure of $C^\infty(\overline{M})$ under the norm
\[
 \|u\|_{W^{1,p}(\rho^{1-2\sigma}, M)}=\left(\int_{M}\rho^{1-2\sigma}(|u|^p+|\nabla u|^p)\,\ud v_g\right)^{\frac 1p},
\]
where $\ud v_g$ denote the volume form of $(M,g)$. $W^{1,p}(\rho^{1-2\sigma}, M)$ is a Banach space for all $p\in [1,+\infty)$ (see \cite{K}). The following Proposition follows directly from Theorem 8.8 and Theorem 8.12 in \cite{GO}.

\begin{prop}\label{sharp sobolev embedding}
Let $\Omega$ be a bounded domain in $\R^{n+1}$ with Lipschitz boundary $\pa\Omega$. Let $\sigma\in (0,1)$, $1\leq p\leq q<\infty$ with $\frac{1}{n+1}>\frac{1}{p}-\frac{1}{q}$ and $d(x)$ be the distance from $x$ to $\pa\Omega$.

(i) Suppose that $2-2\sigma\le p$. Then $W^{1,p}(d^{1-2\sigma},\Omega)$ is compactly embedded in $L^q(d^{1-2\sigma},\Omega)$ if
\[
\frac{2-2\sigma}{p(n+2-2\sigma)}>\frac{1}{p}-\frac 1q.
\]

(ii) Suppose that $2-2\sigma> p$. Then $W^{1,p}(d^{1-2\sigma},\Omega)$ is compactly embedded in $L^q(d^{1-2\sigma},\Omega)$ if and only if
\[
\frac{1}{n+2-2\sigma}>\frac{1}{p}-\frac 1q.
\]
\end{prop}

\begin{cor}\label{SSE on manifold}
For $n\geq 2$,  let $(M,g)$ be an $n+1$ dimensional, compact, smooth Riemannian manifold with smooth boundary $\pa M$.
Let $\sigma\in (0,1)$, and $\rho$ be a defining function of $M$ with $|\nabla_g \rho|=1$ on $\pa M$. Let $1\leq p\leq q<\infty$ with $\frac{1}{n+1}>\frac{1}{p}-\frac{1}{q}$.

(i) Suppose that $2-2\sigma\le p$. Then $W^{1,p}(\rho^{1-2\sigma},M)$ is compactly embedded in $L^q(d^{1-2\sigma},M)$ if
\[
\frac{2-2\sigma}{p(n+2-2\sigma)}>\frac{1}{p}-\frac 1q.
\]

(ii) Suppose that $2-2\sigma> p$. Then $W^{1,p}(d^{1-2\sigma}, M)$ is compactly embedded in $L^q(d^{1-2\sigma}, M)$ if and only if
\[
\frac{1}{n+2-2\sigma}>\frac{1}{p}-\frac 1q.
\]
\end{cor}
\begin{proof}
It follows from Proposition \ref{sharp sobolev embedding} and partition of unity.
\end{proof}

\begin{prop}\label{poincare inequality}
For $n\geq 2$,  let $(M,g)$ be an $n+1$ dimensional, compact, smooth Riemannian manifold with smooth boundary $\pa M$.
Let $\sigma\in (0,1)$, $\rho$ be a defining function of $M$ with $|\nabla_g \rho|=1$ on $\pa M$, and $(u)_{M,\rho}=\int_M\rho^{1-2\sigma}u\,\ud V_g/\int_M\rho^{1-2\sigma}\ud V_g$. Let $1<p<\infty$. Then there exists a constant $C$, depending only on $M,g,p,n,\sigma$ and $\rho$, such that
\be\label{eq:poincare}
\|u-(u)_{M,\rho}\|_{L^p(\rho^{1-2\sigma}, M)}\leq C \|\nabla_g u\|_{L^p(\rho^{1-2\sigma}, M)}
\ee
for every function $u\in W^{1,p}(\rho^{1-2\sigma},M)$.
\end{prop}

\begin{proof}
We argue by contradiction. Were the stated estimate false, there would exist for each integer $k=1,2,\cdots$ a function 
$u_k\in W^{1,p}(\rho^{1-2\sigma},M)$ satisfying
\[
\|u_k-(u_k)_{M,\rho}\|_{L^p(\rho^{1-2\sigma}, M)}>k \|\nabla_g u_k\|_{L^p(\rho^{1-2\sigma}, M)}.
\]
For each $k$, define
\[
v_k:=\frac{u-(u)_{M,\rho}}{\|u-(u)_{M,\rho}\|_{L^p(\rho^{1-2\sigma}, M)}}.
\]
Then
\[
(v_k)_{M,\rho}=0,\quad \|v_k\|_{L^p(\rho^{1-2\sigma}, M)}=1,\quad \|\nabla_g v_k\|_{L^p(\rho^{1-2\sigma}, M)}<1/k.
\]
By Corollary \ref{SSE on manifold}, there exists a subsequence of $\{v_k\}$, which is still denoted as $\{v_k\}$, and a function
$v\in L^p(\rho^{1-2\sigma},M)$ such that
\[
v_k\to v \ \mbox{ in }\  L^p(\rho^{1-2\sigma},M), \quad v_k\rightharpoonup v \ \mbox{ in }\  W^{1,p}(\rho^{1-2\sigma},M).
\]
Consequently, 
\[
(v)_{M,\rho}=0,\quad \|v\|_{L^p(\rho^{1-2\sigma}, M)}=1, \quad \|\nabla_g v\|_{L^p(\rho^{1-2\sigma}, M)}\le\liminf_{k\to\infty}\|\nabla_g v_k\|_{L^p(\rho^{1-2\sigma}, M)}=0.
\]
We reach a contradiction.
\end{proof}

\begin{cor}\label{poincare inequality2}
For $n\geq 2$,  let $(M,g)$ be an $n+1$ dimensional, compact, smooth Riemannian manifold with smooth boundary $\pa M$.
Let $\sigma\in (0,1)$, $\rho$ be a defining function of $M$ with $|\nabla_g \rho|=1$ on $\pa M$, and $(u)_{M,\rho}=\int_M\rho^{1-2\sigma}u\,\ud V_g/\int_M\rho^{1-2\sigma}\ud V_g$. Let $1<p<\infty$. Then there exists a constant $\delta_0$ depending only on $n, \sigma, p$ such that for any $1\le k\le 1+\delta_0$, 
\be\label{eq:poincare2}
\|u-(u)_{M,\rho}\|_{L^{kp}(\rho^{1-2\sigma}, M)}\leq C \|\nabla_g u\|_{L^p(\rho^{1-2\sigma}, M)}
\ee
for every function $u\in W^{1,p}(\rho^{1-2\sigma},M)$, where $C$ is a positive constant depending only on $M,g,p,n,\sigma$ and $\rho$,
\end{cor}

\begin{proof}
By Corollary \ref{SSE on manifold}, there exists a constant $\delta_0$ depending only on $n, \sigma, p$ such that for any $1\le k\le 1+\delta_0$, 
\[
\begin{split}
\|u-(u)_{M,\rho}\|_{L^{kp}(\rho^{1-2\sigma}, M)}&\leq C \|\nabla_g u\|_{L^p(\rho^{1-2\sigma}, M)}+ C \| u-(u)_{M,\rho}\|_{L^p(\rho^{1-2\sigma}, M)}\\
& \leq C \|\nabla_g u\|_{L^p(\rho^{1-2\sigma}, M)}
\end{split}
\]
where in the last inequality we have used Proposition \ref{poincare inequality}.
\end{proof}

Let $(M,g)$, $\rho$ be as in Theorem \ref{thm: main thm A}. For $\sigma\in (0,1)$, we consider 
\be\label{nb1}
\begin{cases}
& \mathrm{div}_g(\rho^{1-2\sigma} \nabla_g u)=0,\quad \mbox{in } M\\
&\lim_{y\to x\in \pa M}\rho(y)^{1-2\sigma}\frac{\pa_g u}{\pa \nu}=f(x) \quad \mbox{on }\pa M.
\end{cases}
\ee
We say $u\in W^{1,1}(\rho^{1-2\sigma}, M)$ is a weak solution of \eqref{nb1} if
\be\label{eq:dwsm}
\int_{M}\rho^{1-2\sigma}\langle\nabla_g u,\nabla_g \varphi \rangle\,\ud v_g=\int_{\pa' M}f\varphi \,\ud s_g
\ee
for all $\varphi\in C^{\infty}(\overline M)$.
Define $\tilde{H}^1:=\{u\in H^1(\rho^{1-2\sigma}, M):\int_{M}\rho^{1-2\sigma}u\,\ud v_g=0\}$. 
\begin{lem}\label{lem:nb1}
 Let $f\in H^{-\sigma}(\pa M):=(H^{\sigma}(\pa M))^*$, the dual of $H^{-\sigma}(\pa M)$, such that 
$\langle f,1\rangle=0$. Then \eqref{nb1} admits a unique weak solution $u\in \tilde{H}^1$.
\end{lem}
\begin{proof}
 The lemma follows immediately from Proposition \ref{poincare inequality} and the Lax-Milgram theorem.  
\end{proof}

\begin{lem} \label{lem:nb2}
 Let $f\in L^2(\pa M)$ with zero mean value, $u\in \tilde{H}^1$ be the weak solution of \eqref{nb1}. Then for any $\theta>1$,
\[
 \int_{M}\rho^{1-2\sigma}\frac{|\nabla_g u|^2}{(1+|u|)^\theta}\,\ud v_g\leq \frac{1}{\theta-1} \|f\|_{L^1(\pa M)}. 
\]
\end{lem}
\begin{proof}
 In our proofs of this and the next lemma, we adapt some arguments from \cite{BG} and \cite{GS}. For $\theta>0$, let 
$\phi_\theta(r)= \int_{0}^r\frac{\ud t}{(1+t)^\theta}$if $r\geq 0$ and $\phi_\theta(r)=-\phi_\theta(-r)$ if $r<0$. It is easy to see that $\varphi_{\theta}:=\phi_{\theta}(u)\in H^1(\rho^{1-2\sigma}, M)$ and 
$|\varphi_\theta|\leq 1/(\theta-1)$ on $\overline M$ if $\theta>1$. Hence, the Lemma follows from multiplying \eqref{eq:dwsm} by letting $\varphi=\varphi_\theta$.  
\end{proof}

\begin{lem}\label{lem:nb3}
  Let $f\in L^2(\pa M)$ with zero mean value, $u\in \tilde{H}^1$ be the weak solution of \eqref{nb1}.  Then there exists $\va_0>0$ depending only on $n$ and $\sigma$ such that for any $1\leq \tau\leq 1+\va_0$, we have
\[
 \|u\|_{W^{1,\tau}(\rho^{1-2\sigma}, M)}\leq C,
\]
where $C>0$ depends only on $M,g, \sigma, \rho, \|f\|_{L^1(\pa M)}$.  
\end{lem}
\begin{proof}
By the H\"older inequality, 
\be\label{nb2}
 \begin{split}
  &\int_{M}\rho^{1-2\sigma}|\nabla_g u|^\tau\,\ud v_g\\&
\leq  \left( \int_{M}\rho^{1-2\sigma}\frac{|\nabla_g u|^2}{(1+|u|)^\theta}\,\ud v_g\right)^{\tau/2}\left(\int_{M}\rho^{1-2\sigma}(1+|u|)^{\frac{\tau \theta}{2-\tau}}\ud v_g\right)^{(2-\tau)/2}\\
&\leq C(\theta)\left(\int_{M}\rho^{1-2\sigma}(1+|u|)^{\frac{\tau \theta}{2-\tau}}\ud v_g\right)^{(2-\tau)/2},
 \end{split}
\ee
where we used Lemma \ref{lem:nb2} in the last inequality and $\theta\in (1,2)$ will be chosen later. Applying Corollary \ref{poincare inequality2} (see also \cite{FKS1}) to  $\varphi_{\theta/2}$ yields that for any $1\leq k\leq 1+\delta_0$
\be\label{nb2-1}
\left(\int_{M}\rho^{1-2\sigma}\big|\varphi_{\theta/2} -\dashint_{M}\rho^{1-2\sigma}\varphi_{\theta/2}\,\ud v_g\big|^{2k}\ud v_g\right)^{1/k}\leq C\int_{M}\rho^{1-2\sigma}\frac{|\nabla_g u|^2}{(1+|u|)^\theta}\,\ud v_g,
\ee
where $\delta_0>0$ depends only on $n, \sigma$, and $C$ depends only on $M,g, \sigma, \rho,$ $k$. 
Since $\phi_{\theta/2}(r) \approx |r|^{1-\frac{\theta}{2}}$ for $|r|$ large, it follows from \eqref{nb2-1} and Lemma \ref{lem:nb2} that
\be\label{nb2-2}
 \left(\int_{M}\rho^{1-2\sigma}|u|^{k(2-\theta)}\right)^{1/2k}\,\ud v_g \leq C+ C \int_{M}\rho^{1-2\sigma}|u|^{1-\frac{\theta}{2}}\,\ud v_g.
\ee
Choosing $\theta$ close to $1$ such that $k(2-\theta)=\frac{\tau\theta}{2-\tau}$ (this can be achieved as long as $\tau$ is closed to $1$) and inserting \eqref{nb2-2} to \eqref{nb2}, we obtain 
\be \label{nb3}
 \begin{split}
  \left(\int_{M}\rho^{1-2\sigma}|\nabla_g u|^\tau\,\ud v_g\right)^{1/\tau}&
\leq C\left(1+ \int_{M}\rho^{1-2\sigma}|u|^{1-\frac{\theta}{2}}\,\ud v_g\right)^{\frac{\theta}{2-\theta}}\\&
\leq C+ C\left(\int_{M}\rho^{1-2\sigma}|u|\,\ud v_g\right)^{\frac{\theta}{2}}
 \end{split}
\ee
Since $\int_{M}\rho^{1-2\sigma}u\,\ud v_g=0$, by the Poincar\'e-Sobolev inequality, H\"older inequality and \eqref{nb3}, we have 
\[
 \|u\|_{L^1(\rho^{1-2\sigma}, M)}\leq C \int_{M}\rho^{1-2\sigma}|\nabla_g u|\,\ud v_g\leq C(1+\|u\|^{\frac{\theta}{2}}_{L^1(\rho^{1-2\sigma}, M)}).
\]
Thus, $\|u\|_{L^1(\rho^{1-2\sigma}, M)}\leq C$ because $\frac{\theta}{2}<1$. Therefore, the lemma follows immediately from \eqref{nb3} and the Poincar\'e-Sobolev inequality. 

\end{proof}

\begin{thm}\label{thm:nb1}
For any bounded radon measure $f$ defined on $\pa M$ with $\langle f, 1\rangle =0$, 
there exists a weak solution  $u\in W^{1,1+\va_0}(\rho^{1-2\sigma},M)$ of \eqref{nb1}. 
\end{thm}

\begin{proof}
 The proof follows from Lemma \ref{lem:nb1} and \ref{lem:nb3} and some standard approximating procedure, see, e.g., \cite{GS}. We omit the details here. 
\end{proof}

\begin{thm}\label{thm:nb2} For $x_0\in \pa M$, let $f=\delta_{x_0}-\frac{1}{|\pa M|_g}$, where $|\pa M|_g$ is the area of $\pa M$ with respect to the induced metric $g$.
Then there exists a weak solution $u\in W^{1,1+\va_0}(\rho^{1-2\sigma},M)$ of \eqref{nb1} with mean value zero and for all $x\in \overline{M}\backslash\{x_0\}$,
\begin{subequations}
\begin{align}
 A_1 \mathrm{dist}_g(x, x_0)^{2\sigma-n}-A_0\leq u(x)&\leq A_2 \mathrm{dist}_g(x, x_0)^{2\sigma-n},\label{nb4}\\
 |\nabla_{tan} u|&\le A_3 \mathrm{dist}_g(x, x_0)^{2\sigma-n-1}, \label{nb4.1}\\
 |\frac{\pa u}{\pa \nu}|&\le A_4\rho^{2\sigma-1}\mathrm{dist}_g(x, x_0)^{-n}, \label{nb4.2}
\end{align}
\end{subequations}
where $A_0, A_1, A_2, A_3, A_4$ are positive constants depending only on $M, g, n, \sigma, \rho$.
\end{thm}

\begin{proof} Let $f_k\in C^1(\pa M)$ with $\int_{\pa M}f_k\,\ud s_g=0$, $\|f_k\|_{L^1(\pa M)}\leq C$ independent of $k$,  such that $f_k \to f$ in distribution sense
as $k\to \infty$. We can also 
assume that $f_k\to f$ in $C^1_{loc}(\pa M\setminus \{x_0\})$. 
By Lemma \ref{lem:nb1} and Lemma \ref{lem:nb3}, there exists a unique solution $u_k\in \tilde{H}^1$ of \eqref{nb1} with $f$ replaced by $f_k$, and 
\[
\|u_k\|_{W^{1,1+\va_0}(\rho^{1-2\sigma},M)}\leq C(\|f_k\|_{L^1(\pa M)})\leq C.
\]
Moreover, it follows from Moser's iterations (see, e.g., the proof of Theorem \ref{thm: L-infty}) that 
there exists some $\al>0$ such that
\be\label{nb5}
 \|u_k\|_{C^\al(M\setminus \mathcal{B}_{r}(x_0))}\leq C(r)
\ee
for any $r>0$. By standard compactness arguments, $u_k\rightharpoonup u$ in $W^{1,1+\va_0}(\rho^{1-2\sigma},M)$ for some $u$, which is a weak solution of \eqref{nb1} and satisfies
\[ 
\|u\|_{C^{\al/2}(M\setminus \mathcal{B}_{r}(x_0))}\leq C(r).
\]
Now, it suffices to establish the estimate \eqref{nb4} for $x\in B_r(x_0)$. For $r$ suitably small, choose a Fermi coordinate system $\{y_1, \cdots, y_{n+1}\}$ centered at $x_0$. Then 
$u_k(y)$ satisfies 
\[
 \begin{cases}
\pa_i(\rho^{1-2\sigma}\sqrt{\det g}g^{ij}\pa_ju_k)=0,&\quad \mbox{in } \mathcal{B}_{2r}^+,\\
-\dlim_{y_{n+1}\to 0}\rho^{1-2\sigma}\sqrt{\det g}\frac{\pa u_k}{\pa y_{n+1}}=f_k,&\quad \mbox{on }\pa' \mathcal{B}_{2r}^+.  
\end{cases}
\]
Let $v_k$ be the unique weak solution of
\[
 \begin{cases}
\pa_i(\rho^{1-2\sigma}\sqrt{\det g}g^{ij}\pa_jv_k)=0,&\quad \mbox{in } \mathcal{B}_{2r}^+,\\
-\dlim_{y_{n+1}\to 0}\rho^{1-2\sigma}\sqrt{\det g}\frac{\pa v_k}{\pa y_{n+1}}=-\frac{1}{|\pa M|},&\quad \mbox{on }\pa' \mathcal{B}_{2r}^+,\\
v_k=u_k &\quad \mbox{on }\pa'' \mathcal{B}_{2r}^+.
\end{cases}
\]
in $H^1(\rho^{1-2\sigma}, M)$. In view of \eqref{nb5}, $\|v_k\|_{L^\infty(\mathcal B_{2r})}\leq C(r)$ and hence $\|v_k\|_{C^\al(\mathcal B^+_{r})}\leq C(r)$. Moreover, $w_k:=u_k-v_k\in H^1(\rho^{1-2\sigma}, M)$ satisfies
\[
 \begin{cases}
\pa_i(\rho^{1-2\sigma}\sqrt{\det g}g^{ij}\pa_jw_k)=0,&\quad \mbox{in } \mathcal{B}_{2r}^+,\\
-\dlim_{y_{n+1}\to 0}\rho^{1-2\sigma}\sqrt{\det g}\frac{\pa w_k}{\pa y_{n+1}}=f_k+\frac{1}{|\pa M|},&\quad \mbox{on }\pa' \mathcal{B}_{2r}^+,\\
w_k=0 &\quad \mbox{on }\pa'' \mathcal{B}_{2r}^+.
\end{cases}
\]
Recall that $g^{i,n+1}=0$ for $i<n+1$ on $\pa'\mathcal{B}_{2r}^+$. Let $\bar w_k$ be the even extension of $w_k$ in $\mathcal{B}_{2r}$, i.e., 
\[\bar w_k=
\begin{cases}
w_k(y',y_{n+1}),&\quad y_{n+1}\geq 0,\\
w_k(y',-y_{n+1}), &\quad y_{n+1}\leq 0. 
 \end{cases}
\]
We also evenly extend $g$ and $\rho$ to be $\bar g$ and $\bar \rho$, respectively. It is easy to verify that the weak limit $w$ of $\bar w_k$ in $L^{1+\va_0}(\rho^{1-2\sigma}, \mathcal{B}_{2r})$ is the \emph{weak solution vanishing on $\pa \mathcal{B}_{2r}$} (see page 162 of \cite{FJK2}) of
\[
 \pa_i(\bar\rho^{1-2\sigma}\sqrt{\det \bar g}\bar g^{ij}\pa_j w)= -2\delta_{0}\quad \mbox{in }\mathcal{B}_{2r}.
\]
It follows from Theorem 3.3 of \cite{FJK2} that $w$ satisfies the estimates \eqref{nb4} in $\B_r(x_0)$. Thus, $u$ satisfies \eqref{nb4}. Finally, \eqref{nb4.1} and \eqref{nb4.2} follows from \eqref{nb4}, Theorem \ref{thm: schauder}, Proposition \ref{thm: schauder-1} and some scaling arguments.
\end{proof}

\bigskip

\noindent Tianling Jin

\noindent Department of Mathematics, Rutgers University\\
110 Frelinghuysen Road, Piscataway, NJ 08854, USA\\[0.8mm]
Email: \textsf{kingbull@math.rutgers.edu}

\medskip

\noindent Jingang Xiong

\noindent School of Mathematical Sciences, Beijing Normal University\\
Beijing 100875, China\\[1mm]
\noindent and\\[1mm]
\noindent Department of Mathematics, Rutgers University\\
110 Frelinghuysen Road, Piscataway, NJ 08854, USA\\[0.8mm]
Email: \textsf{jxiong@mail.bnu.edu.cn/jxiong@math.rutgers.edu}

\end{document}